\documentclass[10pt,letter]{amsart}
\usepackage[utf8]{inputenc}
\usepackage{amsmath, calligra, mathrsfs}
\usepackage[pagebackref,hidelinks]{hyperref}
\usepackage[all]{xy}
\usepackage{amsfonts} 
\usepackage{amssymb} 
\usepackage{graphics} 
\usepackage{amscd} 
\usepackage{enumitem} 
\usepackage{todonotes} 
\usepackage{tikz-cd}
\usepackage[OT2,T1]{fontenc}
\usepackage{mathtools}
\usepackage{xcolor}
\usepackage{chngcntr}
\usepackage{textcomp}
\definecolor{winered}{RGB}{136, 0, 21}

\counterwithin{equation}{subsection}

\hypersetup{
  colorlinks   = true, 
  urlcolor     = teal, 
  linkcolor    = teal, 
  citecolor   = teal 
}

\usepackage{stmaryrd}
\newtheorem{thm}{Theorem}[section]
\newtheorem{prop}[thm]{Proposition}
\newtheorem{cor}[thm]{Corollary}
\newtheorem{lem}[thm]{Lemma}

\newtheorem*{thmen}{Theorem}

\theoremstyle{definition}
\newtheorem{rmk}[thm]{Remark}
\newtheorem{defi}[thm]{Definition}

\DeclareMathOperator{\bA}{\mathbb{A}}

\DeclareMathOperator{\bF}{\mathbb{F}}

\DeclareMathOperator{\Z}{\mathbb{Z}}

\newcommand{\cE}{\mathcal{E}}

\newcommand{\fp}{\mathfrak{p}}

\newcommand{\fX}{\mathfrak{X}}

\newcommand{\ord}{\mathrm{ord}}
\newcommand{\can}{\mathrm{can}}
\newcommand{\naive}{\operatorname{naive}}

\title{Higher Hida  theory for Drinfeld modular curves}
\author{Daniel Barrera Salazar} 
\address{Daniel Barrera \newline Departamento de Matem\'atica y Ciencia de la Computaci\'on, Universidad de Santiago de Chile, Las Sophoras 173, Estación Central, Santiago} 
 \thanks{D.BS.'s research was partly funded by ECOS230025, Anid Fondecyt Regular grant 1241702 and a Lluís Santaló Visiting Position funded by the CRM in Barcelona}
\email{daniel.barrera.s@usach.cl}

\author{Héctor del Castillo} \thanks{H.dC.'s research was supported by ANID Postdoctoral Project 3220656 and a Mitacs Globalink Research Award FR115169.}
\address{H\'ector de Castillo \newline Departamento de Matem\'atica y Ciencia de la Computaci\'on, Universidad de Santiago de Chile, Las Sophoras 173, Estación Central, Santiago}
\email{hector.math@gmail.com}

\author{Giovanni Rosso}\thanks{G.R.'s research was partly funded by the NOVA-FRQNT-CRSNG grant 325940, the NSERC grant RGPIN-2018-04392  and a Mitacs Globalink Research Award FR115169.}
\address{Giovanni Rosso \newline Department of Mathematics and Statistics, Concordia University, 1455 De Maisonneuve Blvd. W.,
Montreal, QC  H3G 1M8
Canada }
\email{giovanni.rosso@concordia.ca}

\begin{document}
\begin{abstract}
 Inspired by the construction of Higher Hida theory of Boxer and Pilloni, we develop Higher Hida theory for the cohomology of the line bundles of Drinfeld modular forms on the Drinfeld modular curve. We also interpolate Serre duality.
\end{abstract}
\maketitle
\tableofcontents

\section{Introduction}

Hida theory studies the deformation of ordinary modular forms and has been developed in great generality for several different types of automorphic forms over number fields. It is a degree zero cohomology theory. Higher Hida theory instead develops the deformations of higher degree cohomology classes: although very recent, it has produced several important arithmetic applications such as in modularity results as in \cite{BGGP} or Bloch--Kato conjectures as in \cite{LoefflerZerbes}.  

Boxer and Pilloni have developed higher Hida theory for modular curves in \cite{BoPi2022} and they were able to obtain this theory for more general situations, such as for Siegel Shimura varieties, in \cite{BoPi2023}.

This paper wants to study deformation not of modular forms, but of Drinfeld modular forms, a function field analogue of modular forms. So far, the theory has been developed only for degree zero coherent cohomology groups, see \cite{NiRo2021}. Our goal is to construct, in tandem, Hida theory in degree zero and higher Hida theory in degree one for the Drinfeld modular curve with a specific tame level structure.
 We follow mainly the constructions  of \cite{BoPi2022} with some {\it ad hoc} modifications to deal with the fact that our Iwasawa algebra is not Noetherian, see for example \ref{lem:iwalg} and that the Drinfeld module of rank two are not necessarily auto dual, see for example Section \ref{sec:dualityfamilies}. 

We expect this will open new directions in the function field case and be a baby case, for example, for the study of higher Hida theory for the general linear groups, or higher Coleman theory.\\  

Now, we describe the results and tools used to prove our results. Fix a prime number $p$ and $q$ a power of $p$. Let $A=\mathbb F_q[T]$ and $\mathfrak p$ be a prime ideal of $A$. Denote by $A_{\mathfrak p}$ the $\mathfrak p$-completion of $A$. Let $X\to \operatorname{Spec}(A_{\mathfrak p})$ be the compactified Drinfeld modular curve of level $\Gamma^\Delta_1(\mathfrak n)$ over $A_{\mathfrak p}$. Let $\omega$ be the Drinfeld modular forms line bundle of $X$ and for an integer $k\in \mathbb Z$ write $\omega^k$ for $\omega^{\otimes k}$. We also denote $\omega_D$ the dual Drinfeld modular forms bundle of $X$ (see Section \ref{subsec:Serredual}).

Let $\Lambda=A_{\mathfrak p}\llbracket A_{\mathfrak p}^\times\rrbracket$ be the Iwasawa algebra. Each integer $k\in \mathbb Z$ defines the character $x\mapsto x^k$ of $A_{\mathfrak p}$, and thus a morphism $k \colon \Lambda \to A_{\mathfrak p}$.  The main results of this paper, Theorems \ref{thm:projandinter} and \ref{thm:dualityfamilies}, are summarized in the following statement.
\begin{thmen}
    There are two finite type $\Lambda$-modules $M$ and $N$ carrying an action of the Hecke algebra of level prime to $p$ such that there are canonical Hecke equivariant isomorphisms for all $k\geq 3$:  \[M\otimes_{\Lambda,k}A_\fp\cong e(T_\mathfrak p)\operatorname{H}^0(X,\omega^k),\]
    \[N \otimes_{\Lambda,k}A_\fp\cong e(T_\mathfrak p)\operatorname{H}^1(X,\omega^{1-k}\otimes\omega_D(-2D)),\]
    where $D$ is the cusps divisor and $e(T_\mathfrak p)$ is the projector of a Hecke operator $T_\mathfrak p$. Furthermore, There is a  perfect pairing $M\times N \to \Lambda$ which interpolates the Serre duality pairing.
\end{thmen}

To obtain this result, we closely follow the ideas for the modular curve in \cite{BoPi2022}. First we define the Hecke operator $T_\mathfrak p$ as a cohomological correspondence and study its properties. In particular, we do a local analysis of $T_\mathfrak p$ at supersingular points and its  Serre-dual operator.  Secondly, we use the theory of Hodge--Tate--Taguchi (see \cite{Ha2020dual} and \cite{Tag1995}) to certain canonical subgroups of \emph{ordinary} Drinfeld modules to interpolate the modular line bundle. Indeed, modulo $\mathfrak p^n$ there is a canonical $A$-submodule $H^{\can}_n$ over $X_n^{\ord}$ of the universal Drinfeld module.
Here $X_n^{\ord}$ is the ordinary locus of $X_n=X\times \operatorname{Spec}(A/\mathfrak p^n)$. It has the property that its Cartier--Taguchi dual $(H_n^{\can})^D$ is an \'etale group \'etale locally isomorphic to $A/\mathfrak p^n$. By taking the $\mathfrak p$-adic limit, we can thus define the Igusa tower as the torsor of frame of $\lim (H_n^{\can})^D$: 
\[\operatorname{Ig}=\operatorname{Isom}_{\mathfrak X^{\ord}}(A_\fp,\lim (H_n^{\can})^D)\to \mathfrak X^{\ord},\]
where $\mathfrak X^{\ord}=\operatorname{colim}X_n^{\ord}$.
On the other hand, the Hodge--Tate--Taguchi map  induces, again $\mathfrak p$-adically, the following isomorphism
\[\omega\cong\lim (H_n^{\can})^D\otimes\mathcal O_{\mathfrak X^{\ord}}\cong(\mathcal{O}_{\operatorname{Ig}}\widehat\otimes A_\fp )^{A_\fp^\times},\]
where the $A_\mathfrak p^{\times}$-action is taken diagonally. This description allows us to define the universal family: let $\kappa^{un} \colon A_\mathfrak p^\times \to \Lambda^\times$ be the universal group homomorphism,  we define 
\[\omega^{\kappa^{un}} \coloneqq (\mathcal{O}_{\operatorname{Ig}}\widehat\otimes\Lambda )^{A_\fp^\times}, \]
where the action is taken diagonally, with the action on $\Lambda$ is given by $\kappa^{un}$. This is an invertible $\mathcal{O}_{\mathfrak X^{\ord}}\widehat \otimes \Lambda$-module.

Now that we have interpolated the modular forms bundles $\omega^k$, we are ready to do the next step: the construction of the modules $M$ and $N$. This is done in two stages that we will describe now. First, we define two auxiliary modules using the theory of compact supported coherent cohomology.  This theory has various point of view on how to define it. We are going to follow the approach using condensed mathematics \cite{ClSc2019}, since it give us access to the six operations formalism. This give a smoother approach and have been applied successfully  in \cite{BoPi2023}.  Our main object of study are the following $\Lambda$-modules
\begin{align*}
  \operatorname{H}^0(\mathfrak X^{\ord},\omega^{\kappa^{un}}), \quad \operatorname{H}^1_c(\mathfrak X^{\ord},\omega^{\kappa^{un}}). 
\end{align*}
The definition of $\operatorname{H}^1_c$ is where we use condensed mathematics. First write $\Lambda = \lim R_n$, where the $R_i's$ are finite rings. Denote by $I_n$ the kernel of $\Lambda \to R_n$. Now, we use the formalism of six operations. In particular there is $q_{n,!}$, where $q_n\colon X_n^{\ord}\to \operatorname{Spec}(A/\mathfrak p^n)$ is the structure map of the ordinary locus. Then $\operatorname{H}^1_c(X^{\ord}_n,\omega^{\kappa^{un}}\otimes (\Lambda/I_n))\coloneqq \operatorname{H}^1(\operatorname{Spec}(A/\mathfrak p^n), q_{n,!}\omega^{\kappa^{un}}\otimes (\Lambda/I_n))$. Remark that via \cite[Prop. 2.3.5]{BoPi2023} this definition is exactly the same used in \cite[\S4.2.2]{BoPi2022}, where the authors used a more classical approach not using condensed mathematics. Finally, we define
\[\operatorname{H}^1_c(\mathfrak X^{\ord},\omega^{\kappa^{un}})=\lim \operatorname{H}^1_c(X^{\ord}_n,\omega^{\kappa^{un}}\otimes (\Lambda/I_n)).\]

The final stage is to construct two operators $U_\mathfrak p$ and $F$ on these cohomology, such that they have projectors $e(U_\mathfrak p)$ and $e(F)$. The construction of $U_\mathfrak p$ and $F$ are analogous as the case of the modular curve. The fact that there exists the projector have to be suitable adapted because the ring $\Lambda$ is non-noetherian. Once we have these projectors, we can finally define the module $M$ and $N$ as $e(U_\mathfrak p)\operatorname{H}^0(\mathfrak X^{\ord},\omega^{\kappa^{un}})$ and $e(F)\operatorname{H}^1_c(\mathfrak X^{\ord},\omega^{\kappa^{un}}\otimes_{\Lambda,d}\Lambda(-2D))$, respectively, where $d\colon \Lambda \to \Lambda$ is the algebra homomorphism induced by the character $A_\mathfrak p^\times\to \Lambda^\times, t\mapsto t^2(\kappa^{un}(t))^{-1}$.

Having construed the module, Theorem will follows form some further properties of the object involved in the construction. Indeed, from the construction one can check a relation between $U_\mathfrak p$ and $F$ with $T_\mathfrak p$ that gives that
\begin{align*}
e(U_\mathfrak p)\operatorname{H}^0(\mathfrak X^{\ord},\omega^{\kappa^{un}})\otimes_{\Lambda,k}A_\fp&=e(T_\mathfrak p)\operatorname{H}^0(X,\omega^k),\\
    e(F)\operatorname{H}^1_c(\mathfrak X^{\ord},\omega^{\kappa^{un}}) \otimes_{\Lambda,k}A_\fp&=e(T_\mathfrak p)\operatorname{H}^1(X,\omega^k).    
\end{align*}
From this the first part of Theorem follows. Further analysis in the arguments also leads to the finite type property. Lastly, let us to mention how we work with the perfect pairing. As in the modular curve case, the basic input is the Kodaira--Spencer isomorphism relating $\Omega^1_{X/A_\mathfrak p}$ and $\omega$. We remark that in our case this relation is not the same as in the modular curve case, because there no auto-duality between Drinfeld modules. This is the reason of the appearance of the dual Drinfeld modular forms bundle $\omega_D$ in the statement of the theorem. Nevertheless, this variant of the Kodaira--Spencer isomorphism still allows us to calculate the Serre dual of $T_\mathfrak p$. Using these two facts and using the relation mentioned above between $T_\mathfrak p$ with $U_\mathfrak p$ and $F$, we can deduce the second part of the Theorem.

Along the way, we also give a proof of the function field version of the classical Serre--Tate Coordinates theorem that lift of ordinary elliptic curves $E_0/\mathbb{F}_p$ are parametrized by $\operatorname{Hom}_{\mathbb{Z}_p}(T_\mathfrak pE_0^D(\overline{\mathbb{F}_p})\otimes T_\mathfrak pE_0(\overline{\mathbb{F}_p}), \mathbb{G}_m[\mathfrak p^\infty])$, see Theorem \ref{thm:SerreTate}.
The result is known to experts, but we decided to give a proof as we could not find a proof in the setting of Drinfeld modules.

 We finish by mentioning a few words about our choice of level structure. The level structure is chosen to follow the theory developed in \cite{Ha2020dual,Ha2022cpt} and it is different from the one used in \cite{NiRo2021}. With this level we have a rather explicit description about the cusp of the Drinfeld modular curve that allow us to define various important notions that are used to define the cohomological correspondence, duality and Kodaira--Spencer isomorphism, as well as to apply the Hodge--Taguchi map, etc. 

\subsection*{Acknowledgments} We thanks G.~B\"ockle for encouraging us to write section \ref{s:serre-tate}, and C.~Pagano for discussion on Carlitz polynomials. We thank V. Pilloni for his useful correspondence.

\section{Drinfeld modular curves} \label{sec:pre}

\subsection{Preliminaries}
We recall again the notation introduced in the introduction. Let $A= \mathbb{F}_q[T]$. Throughout this text, we fix monic irreducible polynomial $\varpi\in \mathbb{F}_q[T]$. The ideal generated by $\varpi$ will be denote by $\mathfrak p$ and $d$ its degree which corresponds to the degree of the polynomial $\varpi$. Finally, we denote by  $A_\mathfrak p$ the $\mathfrak p$-completion of $A$, $K_{\fp}$ its fraction field and $k(\fp)$ its residue field. 

In this work, we consider Drinfeld modules $(E, \varphi)$ over $\bF_q$-schemes $S$ (see \cite{Ha2022cpt} for a precise definition), to simplify notation we usually write the pair $(E, \varphi)$ simply as $E$.  Recall that the $\fp$-torsion points of $(E,\varphi)$ are defined by 
\[E[\fp]= \bigcap_{a\in\fp}\ker (\varphi(a)).\]
We denote by $E^D$ the dual Drinfeld module of $E$ (see \cite[Theorem 2.19]{Ha2020dual}).

For each $M \in \mathbb F_q[T]$, let $[M](X)$ be the Carlitz polynomial with coefficients in $\mathbb F_q[T]$. Recall that it is defined by recursion and linearity in the following way: $[1](X) \coloneqq X$, $[T](X) \coloneqq X^q + T X$, and for $n \geq 2$,  $[T^n](X) \coloneqq [T]([T^{n-1}](X))$. We further define the Carlitz(--Hayes) module which is the Drinfeld module, denoted by $C$, whcih is defined through the morphism $\varphi_C\colon A \to \mathbb F_q\{\tau\}; \,\varphi_C(T)\coloneqq\tau+T,$ where $\mathbb F_q\{\tau\}$ is the non-commutative polynomial ring that satisfies $\tau a=a^q\tau$. Remark that $C$ is the unique (formal) Drinfeld module of rank $1$ and good reduction at $\fp$ (see \cite[Section 4.1]{NiRo2021}).

Carlitz module plays the role of the root of unity group in the classical elliptic curve theory. For example, it will use to define duality for Drinfeld modules, which is central to our work (See the proof of the main proposition in \S\ref{s:serre-tate} and remark \ref{rmk:tag-dual}).

\subsection{Drinfeld modular curves}\label{ss:Drinfeld modular curves}

Denote 
\[\Gamma_1(\mathfrak n)=\left\{\gamma\in {\rm GL}_2(A) \colon \gamma\equiv\begin{pmatrix}* & * \\
0 & 1 
\end{pmatrix}\mod \mathfrak n  \right\}\] \[\Gamma_0(\mathfrak p)=\left\{\gamma\in {\rm SL}_2(A) \colon \gamma\equiv\begin{pmatrix}* & * \\
0 & * 
\end{pmatrix}\mod \mathfrak p  \right\}.\]
Let $\xi$ be a non-zero element of $A$ coprime with $\fp$ and $\mathfrak n$ be the ideal  generated by $\xi$. Suppose that there exists a prime factor $\eta$ of $\xi$ of degree prime to $q-1$. Then we can choose a subgroup $\Delta \subset (A/\mathfrak n)^\times$ such that with the natural inclusion $\mathbb F_q^\times \subset (A/\mathfrak n)^\times$ we have $\Delta \bigoplus\mathbb F_q^\times =(A/\mathfrak n)^\times$. 
We consider the compactified Drinfeld modular curve with level $\Gamma^\Delta_1(\mathfrak n)$ over $\mathrm{Spec}(A_{\fp})$, denoted by $q:X\rightarrow \mathrm{Spec}(A_{\fp})$ (see \cite[Section 3]{Ha2022cpt} for details).  We also consider the Drinfeld modular curve with level at $\fp$, i.e. the compactified Drinfeld modular curve  with level $\Gamma_0(\fp)\cap \Gamma^\Delta_1(\mathfrak n)$, denoted by $q: X_0(\fp)\rightarrow \mathrm{Spec}(A_{\fp})$. In a dense open subset of $X$, this is the moduli space that parametrizes rank two Drinfeld modules $E$ over an $A_\mathfrak p$-schemes with $\Gamma^\Delta_1(\mathfrak n)$-structure. Similarly, in a dense open subset of $X_0(\fp)$, it is the moduli space that parametrizes rank two Drinfeld modules $E$ with $\Gamma^\Delta_1(\mathfrak n)$-structure and $\Gamma_0(\mathfrak p)$-structure. As it is key in our work, we recall that a $\Gamma_0(\mathfrak p)$-structure on an Drinfeld module $E$ over an $A_\mathfrak p$-scheme $S$ corresponds to a finite locally free closed $A$-submodule scheme $H$ of $E[\fp]$ of rank $q^d$ over $S$. 
We denote these open subsets by $Y$ and $Y_0(\mathfrak{p})$, respectively, and their boundary divisor by $D$ and $D_0(\mathfrak{p})$, respectively.

These compactifications come with two finite flat morphisms $p_1$ and $p_2$
\[\begin{tikzcd}[column sep=small]
X_0(\fp) \arrow[rr, shift left, "p_1"]
\arrow[rr, shift right, "p_2"']\arrow[rd,"q"'] &  & X\arrow[dl,"q"].\\
 & \operatorname{Spec}(A_{\fp})& 
\end{tikzcd}\]
In a dense open subset these come with an isogeny
\[\pi \colon p_1^*\mathcal{E} \to p_2^*\mathcal{E}, \]
where $\mathcal{E}$ is  the universal rank two Drinfeld module of level $\Gamma^\Delta_1(\mathfrak n)$. Moreover, it parametrizes the isogenies $E \to E/H$, where $H$ is as above and the $p_1$ and $p_2$ can be described as $(E,H)\mapsto E$ and $(E,H)\mapsto E/H$, respectively.

\subsection{On the special fiber}\label{subsec:specialfib}
Throughout this section we let $(E,\varphi)$ be a Drinfeld module over an $A_\mathfrak p$-scheme $S$ of rank two defined over $\operatorname{Spec}(k(\fp))$. 
In this situation, we have that the base change of $E[\mathfrak p]$ to any  geometric point $x$ of $\mathfrak p$  is either 
\[(\fp^{-1}/A) \text{ or } \{0\}.\]
In the first case, we say that the Drinfeld module $E$ is ($\fp$-)\emph{ordinary} and, in the second, ($\fp$-)\emph{supersingular}.  We denote by $X_1$ the reduction modulo $\varpi$ of $X$ and by $X_1^{\ord}$ the ordinary locus of $X_1$, which is a dense open subset (see \cite{HaCh-Fu2020}). Moreover, the ordinary locus of $X_0(\fp)_1$ is defined as the fiber product of  $X_0(\fp)_1$ and $X_1^{\ord}$ through $p_1$, and is denoted by $X_0(\fp)_1^{\ord}$.
More generally, we shall write $X_n, X_n^{\ord}, X_0(\fp)_n^{\ord},\ldots$ the reductions modulo $\varpi^n$.

In order to have a better understanding of these loci, we explore these conditions in terms of the Drinfeld module structure of $E$. Indeed, we can write \cite[Proposition 2.7]{Shas2007}
\begin{equation}\label{eq:accionp}
        \varphi(\varpi)=(a_0+\dots+a_{2d}\tau^{d})\tau^d.
\end{equation}

\begin{defi}Let $E$ be a Drinfeld module over an $A_{\mathfrak p}$-scheme $S$. We let $E^{(\fp)}=E\times_{S,\operatorname{Frob}^d}S
$, where $d$ is the degree of $\mathfrak p$. We define the Frobenius isogeny $F_{\fp,E}$ to be the morphism
\[\tau^d : E \to E^{(\mathfrak p)}.\]
We also define the Verschiebung $V_{\fp,E}$ by 
\[a_0+\dots+a_{2d}\tau^{d}\colon  E^{(\mathfrak p)}\to E.\]
\end{defi}
When the context is clear, we simply write \( F_{\fp,E} \) and \( V_{\fp,E} \) as \( F \) and \( V \), respectively. 
 \begin{rmk}  We have the exact sequence \cite[2.8]{Shas2007}
\begin{equation} \label{eq:torison}
    0\to \ker F \to E[\fp] \to \ker V\to 0.
\end{equation}
 \end{rmk}



When $S=\operatorname{Spec}(R)$ and $E$ is a point in $X_0(\fp)_1$, then the associated locally free $A$-module of $E[\fp]$ of rank $q^d$ is either isomorphic to $(\fp^{-1}/A)_R$  or $\alpha_{d,R}=\operatorname{Spec}(R[T]/T^{q^{d}})$. Thus, in the ordinary locus $X_0(\fp)_1^{\ord}$ there are two open subsets $X_0(\fp)^{\ord, V}_1$ and $X_0(\fp)^{\ord, F}_1$ corresponding to the two cases above, respectively. Thus, $X_0(\fp)^V_1$ parametrizes the Drinfeld modules $E$ with the $A$-submodule $H$ of order $q^d$ isomorphic to $\ker V$, and  $X_0(\fp)^F_1$ parametrizes the ones isomorphic to $\ker F$. 
Remark that $X_0(\fp)^{\ord, V}_1$ and $X_0(\fp)^{\ord, F}_1$ are both isomorphic to $X^{\ord}_1$ and their union is dense in $X_0(\fp)_1$. We obtain a disjoint union:
\[X_0(\fp)_1^{\ord}=X_0(\fp)^{\ord, F}_1\sqcup X_0(\fp)^{\ord, V}_1.\]
This induces a decomposition into two 
components:
\[X_0(\fp)_1=X_0(\fp)^F_1\cup X_0(\fp)^V_1.\]
Furthermore, for $n\geq 1$ we decompose $X_0(\fp)_n^{\ord}= X_0(\fp)_n^{\ord, F} \sqcup X_0(\fp)_n^{\ord, V}$. Let  $\mathfrak X_0(\mathfrak p)$ be the completion along the closed subscheme $X_0(\fp)_1$ and $ \fX $ be the corresponding formal scheme attached to $X$. We denote by $\fX_0(\fp)^{\mathrm{ord}}$ the $\operatorname{colim_n} X_0(\mathfrak p)_n^{\rm ord}$, and remark that we have the following decomposition
\[\mathfrak X_0(\mathfrak p)^{\ord}=\mathfrak X_0(\mathfrak p)^{\ord,F} \coprod \mathfrak X_0(\mathfrak p)^{\ord,V}.\]

This allows us to describe the morphisms $p_1$ and $p_2$, on these irreducible components  in the following way:

\begin{prop}\label{prop:corr-special}
    In a dense open subset of $X_0(\fp)^F_1$, $p_1$ is an isomorphism, $p_2$ is totally ramified  and $\pi$ identify with the Frobenius isogeny of the universal Drinfeld modules
    \[F\colon\mathcal{E}\to \mathcal{E}^{(\mathfrak p)}.\]
    Similarly, in a dense open subset of $X_0(\fp)^V_1$, $p_1$ is a totally ramified, $p_2$ is an isomorphism and $\pi$ is identified with the Verschiebung isogeny of the universal Drinfeld module
    \[ V\colon \mathcal{E}^{(\mathfrak p)}\to  \mathcal{E}.\]
\end{prop}




\section{On Serre--Tate Coordinates}\label{s:serre-tate}
 In this section we study the projection $p_1: X_0(\fp)\rightarrow X$ to study the integrability of the Hecke correspondence $T_\mathfrak p$ which is crucial is this work (see \S\ref{sec:cohcorr}). We focus on the ordinary locus  of the Drinfeld modular curve which is enough for our purposes.
 
 We denote by $\breve{K}_{\fp}$ the completion of the maximal unramified extension of $K_{\fp}$, by $\breve A_{\fp}$ the ring of integers of $\breve{K}_{\fp}$ and $\breve k$ the residue field of $\breve A_{\fp}$. Finally, let $k$ be an algebraic closure of $k(\mathfrak p)$.
\begin{prop}\label{prop:ordlocal}Let $x\in X_1^{\ord}(k)$ be a geometric point in the ordinary locus.
    \begin{enumerate}
        \item There exist isomorphisms $\widehat X_{(x)}\cong \operatorname{Spec}(\breve A_{\fp}\llbracket X\rrbracket)$ and for all $y\in X_0(\mathfrak p)(k)$ such that $p_1\circ y= x$ we have $\widehat X_0(\mathfrak p)_{(y)}\cong \operatorname{Spec}(\breve A_{\fp}\llbracket X\rrbracket)$. Here $\widehat X_{(x)}$ and $\widehat X_0(\mathfrak p)_{(y)}$ are the completion of the strict henselizations. 
        \item  Let $y\in X_0(\mathfrak p)(k)$ such that $p_1\circ y= x$. We can choose the isomorphisms above such that the following diagram is commutative
        \[
        \begin{tikzcd}
            \widehat X_0(\mathfrak p)_{(y)}\ar[d, "p_1"']\ar[r, "\sim"]&\operatorname{Spec}(\breve A_{\fp}\llbracket X\rrbracket) \ar[d, "f_y"]\\
            \widehat X_{(x)}\ar[r, "\sim"]&\operatorname{Spec}(\breve A_{\fp}\llbracket X\rrbracket),
        \end{tikzcd}\]
        where $f^*_y\colon \breve A_{\fp}\llbracket X\rrbracket \to \breve A_{\fp} \llbracket X \rrbracket$ is either the identity when $y\in X_0(\mathfrak p)^{\ord, F}_1$ or equal to the map sending $X$ to $[\varpi](X)$, which is the Carlitz polynomial, if $y\in X_0(\mathfrak p)^{\ord, V}_1$.    \end{enumerate}
\end{prop}

To obtain such a theorem we start reviewing the Serre--Tate theory for Drinfeld modules.  Let us fix $(E_0,\varphi_0)$ a Drinfeld module of rank two 
defined over $\operatorname{Spec}(k)$. Following \cite{BoBr2019} and \cite{D1974} we recall that we can associate to it a $A_{\fp}$-divisible module \[E_0[\mathfrak p^\infty]=(E[\mathfrak p^n],\iota_n)_n\]
where  $\iota_n$ is the inclusion $E[\mathfrak p^{n}]\hookrightarrow E[\mathfrak p^{n+1}]$, for every positive integer $n$. Furthermore, this $A_{\fp}$-divisible module $E_0[\mathfrak p^\infty]$ has a deformation ring, that we denote by $R_{E_0[\mathfrak p^\infty]}$, that moreover is a power series ring over $\breve A_{\fp}$ in one indeterminate (see \cite{BoBr2019,D1974}). In \cite{D1974} it is proved the following theorem (see \cite[Thm. 2.8]{BoBr2019}):
\begin{thm}(Serre-Tate)\label{SerreTateLocalRing} Let $X_{\breve A_{\fp}}$ be the base change to $\breve A_{\fp}$ of $X$, and $x$ the point in $X_{\breve A_{\fp}}$ corresponding to $(E_0,\varphi_0)$. There is a natural isomorphism of $\breve A_{\fp}$-algebras
\[R_{E_0[\mathfrak p^\infty]}\to \widehat{\mathcal{O}}_{X_{\breve A_{\fp}},x},\]
    where $\widehat{\mathcal{O}}_{X_{\breve A_{\fp}},x}$ is the completion of the local ring $\mathcal{O}_{X_{\breve A_{\fp}},x}$. 
\end{thm}

From this theorem and the remark above we deduce (1) of Proposition \ref{prop:ordlocal}. In the next step we obtain Serre--Tate coordinates from this theorem and following the exposition by Katz in \cite{Ka1981}. In fact, we adapt the classical theory of abelian varieties developed in \cite[\S2]{Ka1981} to the Drinfeld module of rank two case.

Using \cite[Section 4]{Tag1995}, there is a pairing
\[\theta_{\mathfrak p^n}\colon E_0[\mathfrak p^n]\times E_0^D[\mathfrak p^n] \to C ,\]
where is  $C$ is the Carlitz module. Since $E_0$ is ordinary, we have an isomorphism \cite[Lemma 3.6]{Ha2020dual}
\[E^c_0[\mathfrak p^n]\cong \operatorname{Hom}(E_0^D[\mathfrak p^n](k), C[\mathfrak p^n]),\]
where $E^c_0$ is the connected part of $E_0$.
Taking the limit we get a pairing
\[\Theta_0\colon E_0^c[\mathfrak p^\infty]\times T_\mathfrak pE_0^D(k) \to C[\mathfrak p^\infty], \]
where $T_{\mathfrak p}E_0^D(k)=\operatorname{Hom}_{A_{\mathfrak p}} (K_{\mathfrak p}/A_{\mathfrak p}, E_0^D(k))$ and $E_0^D(k)$ has the $A_\mathfrak p$-modules structure coming from $\varphi_0$. In other words, $\Theta_{0}$ induces an isomorphism
\[E^c_0[\mathfrak p^\infty]\cong \operatorname{Hom}_{A_\mathfrak p}(T_\mathfrak pE_0^D(k), C[\mathfrak p^\infty]).\]

We will use these constructions in the proof of the following theorem. First,
let ${\rm Art}_{\breve A_\mathfrak p}$ be the category of artinian local $\breve A_{\mathfrak p}$-algebras with residue field $\breve k$.

\begin{thm}(Serre--Tate Coordinates) \label{thm:SerreTate}The functor ${\rm Art}_{\breve A_\mathfrak p}\to {\rm Sets}$ of deformations of $(E_0,\varphi_0)$ is equivalent to the functor $R\mapsto \operatorname{Hom}_{A_\mathfrak p}(T_\mathfrak p E_0(k)\otimes T_\mathfrak p E_0^D(k), C[\mathfrak p^\infty](R))$.
\end{thm}

In the following, we give the construction of the equivalence in this theorem. Firstly remark that from theorem \ref{SerreTateLocalRing}, it is enough to consider deformations of the  attached $A_{\fp}$-divisible modules. For this let $(R,\mathfrak m)$ be an artinian local algebra over $\breve A_{\fp}$ with residue field $\breve k$. Let $(E,\varphi)$ be a Drinfeld module over $\operatorname{Spec}(R)$ such that modulo $\mathfrak m$ it is isomorphic to $(E_0,\varphi_0)$. Following the procedure as above we obtain a perfect pairing
\[\Theta\colon E^c[\mathfrak p^\infty]\times T_\mathfrak p E_0^D(k) \to C[\mathfrak p^\infty]. \]


\begin{prop}
    There is a homomorphism  $\phi_{E, R}: T_\mathfrak pE_0(k)\rightarrow E^c[\fp^{\infty}](R)$ such that the natural exact sequence 
$$0 \rightarrow E^c[\fp^{\infty}] \to E[\fp^{\infty}] \rightarrow T_\mathfrak pE_0(k)\otimes(K_{\fp}/A_{\fp}) \to 0$$
  comes from the exact sequence 
\begin{align}
    0 \rightarrow T_{\fp}E_0(k) \to T_{\fp}E_0(k)\otimes_{ A_{\fp}} K_{\fp} \rightarrow T_\mathfrak pE_0(k)\otimes_{ A_{\fp}} (K_{\fp}/A_{\fp}) \to 0 \label{modpitorsion}
\end{align}
  via pushing out along $\phi_{E, R}$. 
\end{prop}
 \begin{proof}
We follow  \cite[Appendix, 2.5.5]{MessingCrystals} and \cite[p. 151]{Ka1981}. We proceed in two steps: we show that every extension of $(K_{\fp}/A_{\fp})$ by $E^c[\fp^{\infty}]$ is determined uniquely by a morphism $\phi_{E, R}\colon T_\mathfrak pE_0(k)\rightarrow E^c[\fp^{\infty}](R)$ as above, and then  we construct $\phi_{E, R}$. 

We show that for every extension of $T_\mathfrak pE_0(k)\otimes(K_{\fp}/A_{\fp})$ by $E^c[\fp^{\infty}]$  comes via pushout from a map as above. By identifying $T_{\fp}E_0(k)$ with $A_{\fp}$, it is enough to consider  extension of $K_{\fp}/A_{\fp}$ by $E^c[\fp^{\infty}]$ and its relation with morphisms $A_\mathfrak p\to E^c[\fp^{\infty}] $. We consider the $\fp^{n}$-torsion of \eqref{modpitorsion} that can be written as
\begin{align*}
    0 \rightarrow  A_{\fp} \stackrel{\varpi^n}{\to} A_{\fp}  \rightarrow A_{\fp} /\varpi^n A_{\fp} \to 0
\end{align*}
and we take $\mathrm{Hom}( \phantom{A}, E^c[\fp^{n}])$ and get the long exact sequence
\[
\mathrm{Hom}(A_{\fp}, E^c[\fp^{n}]) \stackrel{\varpi^n}{\to}  \mathrm{Hom}(A_{\fp}, E^c[\fp^{n}]) \to \mathrm{Ext}^1(A_{\fp}/\varpi^n A_{\fp}, E^c[\fp^{n}]) \to \mathrm{Ext}^1(A_{\fp}, E^c[\fp^{n}]).
\]
and as the last term vanishes by flatness of $A_{\fp}$,  we get 
\begin{equation} \label{eq:exthom}
\mathrm{Ext}^1(A_{\fp}/\varpi^n A_{\fp}, E^c[\fp^{n}])  \cong \mathrm{Hom}(A_{\fp}, E^c[\fp^{n}]) / \varpi^n \mathrm{Hom}(A_{\fp}, E^c[\fp^{n}]). 
\end{equation}
Taking the inverse limit, we get that $\mathrm{Ext}^1(K_{\fp}/ A_{\fp}, E^c[\fp^{\infty}]) $ is parametrised by $\mathrm{Hom}(A_{\fp}, E^c[\fp^{\infty}])$.

  For the second step, first note that since $E$ is  Zariski local over $\operatorname{Spec}(R)$ isomorphic to $\mathbb A_R^1$ and $R$ is  local Artinian, we can assume that $E$ is  $\bA_{R}^1$.  We also note that, as $R$ is Artinian, we have $\varpi^{n+1}=0$, for a certain $n\geq 0$. The inverse map of \eqref{eq:exthom}  can be constructed as follows (see (2.6) of \cite{MessingCrystals}):  take any torsion point $x_n$ in $E_0[\mathfrak p^{n}](k)$, and consider a lift $\tilde{x}_n \in E(R)=R$. We claim $[\varpi^n]\tilde{x}_n$ is independent of $\tilde{x}_n$ and belongs to  $E^c[\fp^{\infty}](R)$; indeed take a second lift $\tilde{x}'_n$. Note that the difference $[\varpi^n](\tilde{x}_n - \tilde{x}'_n)$ is $0$ as modulo $\mathfrak m$ they coincide, so $(\tilde{x}_n - \tilde{x}'_n) \in {\mathfrak m} R$, which contains $\varpi$, so $[\varpi^n](\tilde{x}_n - \tilde{x}'_n) \in \varpi^{n+1}R=0$. 
  We now show that it belongs to $E^c(R)$ because of the hypothesis $[\varpi^n]x=0$. Indeed, if we look at $[\varpi^n]\tilde{x}_n \bmod E^c(R)$ we get $[\varpi^n]x_n=0$. This is well defined, independently of $n$, so it extends to $ T_{\fp}E_0(k)$. Thus we constructed a map $\phi_{E, R}\colon T_\mathfrak pE_0(k)\rightarrow E^c[\fp^{\infty}](R)$.
\end{proof}

Now, using $\phi_{E, R}$ we define the following $A_\fp$-bilinear form $q_{E, R}\colon T_\mathfrak p E_0(k)\otimes T_\mathfrak p E_0^D(k)\rightarrow C[\mathfrak p^\infty](R)$ by
$$q_{E, R}(\alpha, \alpha_D)= \Theta(\phi_{E, R}(\alpha), \alpha_D).$$ 

But $q_{E, R}$ is uniquely determined by $\phi_{E, R} \in \operatorname{Hom}_{A_\mathfrak p}(T_\mathfrak p E_0(k), E^c[\mathfrak p^\infty](R))$. In fact, the pairing $\Theta$ induces an isomorphism $E^c[\mathfrak p^\infty](R) \cong \operatorname{Hom}_{A_\mathfrak p}(T_{\mathfrak p} E^D_0(k), C[\mathfrak p^\infty](R))$ and thus we obtain
\[
\operatorname{Hom}_{A_\mathfrak p}(T_\mathfrak p E_0(k), \operatorname{Hom}_{A_\mathfrak p}(T_{\mathfrak p} E^D_0(k), C[\mathfrak p^\infty](R))
)  \cong
\operatorname{Hom}_{A_\mathfrak p}(T_\mathfrak p E_0(k)\otimes T_\mathfrak p E_0^D(k), C[\mathfrak p^\infty](R)
).  \]

To prove Theorem \ref{thm:SerreTate}, we need to show that this  uniquely determines  the lift of $E_0$. Essentially by the classical result of Drinfeld \cite[\S 4, 5]{D1974} on moduli space of Drinfeld modules, to determine a deformation of a Drinfeld module it is enough to deform the associated $\varpi$ divisible group.

We can now complete the proof of the second point of Proposition \ref{prop:ordlocal}: by considering the formal moduli of deformations of $(E_0,\varphi_0)$ and applying Theorem \ref{thm:SerreTate} we get that,   
\[
\widehat{\mathcal{O}}_{{X_0(\mathfrak p)}_{\breve A_{\fp}},y} \cong \operatorname{Hom}_{A_\mathfrak p}(T_\mathfrak p E_0(k)\otimes T_\mathfrak p E_0^D(k), C[\mathfrak p^\infty])
\]
and
\[
\widehat{\mathcal{O}}_{X_{\breve A_{\fp}},x} =\operatorname{Hom}_{A_\mathfrak p}(T_\mathfrak p E'_0(k)\otimes T_\mathfrak p {E'_0}^D(k), C[\mathfrak p^\infty])
\]
where $E'_0$ is either $E_0/\mathrm{Ker}(F)$ when $y\in X_0(\mathfrak p)^{\ord, F}_1$ or $E_0/\mathrm{Ker}(V)$ if $y\in X_0(\mathfrak p)^{\ord, V}_1$. 
Now, note that $\mathrm{Ker}(F)(k)=\left\{ 0 \right\}$, so in the first case we get 
\[
f^*_y:\operatorname{Hom}_{A_\mathfrak p}(T_\mathfrak p E_0(k)\otimes T_\mathfrak p E_0^D(k), C[\mathfrak p^\infty]) \cong \operatorname{Hom}_{A_\mathfrak p}(T_\mathfrak p E'_0(k)\otimes T_\mathfrak p {E'_0}^D(k), C[\mathfrak p^\infty]).
\]

In the second case, $\mathrm{Ker}(V)(k) \cong A_{\fp} /\varpi A_{\fp}$ is the kernel of  multiplication by $\varpi$ on  $E_0(k)[\varpi^{\infty}]$, which proves that $f_y^*$ is $[\varpi](X)$.

\begin{cor} \label{cor:traceord}Let $x\in X^{\ord}_1(k)$ and $y\in X_0(\mathfrak p)(k)$ such that $p_1\circ y= x$, then 
\begin{enumerate}
    \item if $y\in X_0(\fp)^{F}_1$ then the map ${\rm tr}_{p_1}\colon\mathcal{O}_{X_0(\mathfrak p), y}\to 
(p_1^!\mathcal{O}_X)_{y}$  is an isomorphism. 
    \item If $y\in X_0(\fp)^{V}_1$ then the map ${\rm tr}_{p_1}\colon\mathcal{O}_{X_0(\mathfrak p), y}\to 
(p_1^!\mathcal{O}_X)_y$ factors into an isomorphism $\mathcal{O}_{X_0(\mathfrak p), y}\cong  \varpi
(p_1^!\mathcal{O}_X)_y$. 
\end{enumerate}
\end{cor}
\begin{proof}
    Thanks to Proposition \ref{prop:ordlocal} it is enough to study the trace of the map $f_y$. The first case is immediate. The second one is as follows. Consider the $\breve{A}_\mathfrak p\llbracket X\rrbracket$ as a $\breve{A}_\mathfrak p\llbracket X\rrbracket$-module, where the structure is induced by $f_y^*$.  With this, $\{1,\dots, X^{d-1}\}$ is a basis. Since the coefficients of $[\varpi](X)$ beside the leading term are multiple of $\varpi$ and the $X$ term is $\varpi$ \cite[Corollary 2.12]{CCar}, we have that the trace of $f_y^{\ast}$ satisfies that the image of $\breve{A}_\mathfrak p\llbracket X\rrbracket$ is equal to $\varpi\breve{A}_\mathfrak p\llbracket X\rrbracket$. Thus obtaining the desired result. 
\end{proof}

\begin{rmk}
    We remark that for the classical modular curve case, on local coordinates the map $p_1$ on $X_0(p)^{V}_1$ is given by $X \mapsto [p](X)=(1+X)^p-1$. This is the group law of the roots of unity $\mu_{p^\infty}$.

\end{rmk}

\section{Cohomological correspondences} \label{sec:cohcorr}
The purpose of this section is twofold. First, to define the correspondence $T_\mathfrak p$, central in this work, and to prove its integrability. Secondly, to review and apply the Serre duality to this correspondence. For the first point, we use the description of $T_\mathfrak p$ on the Verschiebung and Frobenius locus of $X_0(\fp)_1$ and the Serre--Tate coordinates. For the second point, we use a Kodaira--Spencer isomorphism for Drinfeld modular curves \cite{Ge1990,Tag1995}.

\subsection{\texorpdfstring{The $T_{\fp}$ correspondence} {The Tp correspondence} }\label{subsec:corrTp}
Let $\omega$ be the line bundle of $X$ which gives rise to Drinfeld modular forms (see \S5 in \cite{Ha2022cpt}). In the dense open subset $Y$, defined in \S\ref{ss:Drinfeld modular curves}, it is given by $(\operatorname{Lie}\mathcal{E})^{-1}$ where $\cE$ is the universal object. For $k\in \mathbb Z$, we let $\pi_k:p_2^*\omega^k \to p_1^*\omega^k$ be the rational map induced from $\pi \colon p_1^*\mathcal{E} \to p_2^*\mathcal{E}$. We define the 
unnormalized ``cohomological correspondence''  
\[T_\fp^{\operatorname{naive}}\colon p_2^*\omega^k \to p_1^!\omega^k. \]
as the tensor map of $\pi_k$ and the trace map ${\rm tr}_{p_1}\colon \mathcal{O}_{X_0(\fp)}\to p_1^!\mathcal O_X$. Observe that, from the projection formula we have $p_1^!\omega^k=p_1^*\omega\otimes_{\mathcal{O}_{X_0(\mathfrak p)}}p_1^!\mathcal{O}_X$. We define $T_\fp$ by $\varpi^{-\min(1,k)}T^{\operatorname{naive}}_\mathfrak{p}$.
\begin{prop} \label{prop:Tpint}$T_\fp$ is a morphism of sheaves (optimally) defined over $\operatorname{Spec}(A_\mathfrak p)$.
\end{prop}
\begin{proof} We can check this condition outside a codimension 2 subspace of $X_0(\fp)$ i.e., outside a finite number of points. Outside of the special fiber the map $p_1$ is well behaved and thus it is enough to look at the special fiber. We will work separately on $X_0(\fp)_1^V$ and $X_0(\fp)_1^F$. 

Firstly, in a dense open subset of $X_0(\fp)_1^{\mathrm{ord},F}$, by Proposition \ref{prop:corr-special} the map $\pi_1$ corresponds  to
\[\operatorname{Lie}^{-1}(F): \omega^{q} \to \omega ,\] 
which is the dual of the map on differentials attached to the Frobenius $F$. Therefore, by the definition of the structure of the $A$-module in $\omega$, the equation \eqref{eq:accionp} and the fact that over $X_0(\fp)_1^{\mathrm{ord}}$ the universal object $\cE$ is ordinary if and only if $\operatorname{ker} V$ is \'etale (see \cite[Proposition 2.14]{Shas2007}), we have the following isomorphism
\begin{equation}\label{eq:TpFrob}
    p_2^*\omega_y \xrightarrow{\sim} \varpi (p_1^\ast\omega)_y
\end{equation} 
for each $y$ in the dense open subset of $X_0(\fp)_1^{\mathrm{ord},F}$ mentioned above. By Corollary \ref{cor:traceord}, the trace map  ${\rm tr}_{p_1}\colon\mathcal{O}_{X_0(\mathfrak p)}\to 
p_1^!\mathcal{O}_X$ is an isomorphism. Hence, $T^{\operatorname{naive}}_\mathfrak{p,\xi}$ induces an isomorphism 
\begin{equation}\label{e:T_p in F}
    (p_2^*\omega^k)_y \to \varpi^k(p_1^{!}\omega^k)_y
\end{equation}
From this we deduce the proposition in this case. 

Now, over $X_0(\fp)_1^{\mathrm{ord},V}$, $\pi_1$ corresponds to $\operatorname{Lie}^{-1}(V)$  by Proposition \ref{prop:corr-special}. Moreover, remark that as before $\operatorname{Lie}^{-1}(V)$  is an isomorphism since $V$ is \'etale by equation \eqref{eq:accionp}.  Finally, by Corollary \ref{cor:traceord} we have that the trace morphism ${\rm tr}_{p_1}$ factors into an isomorphism
\[\mathcal{O}_{X_0(\fp),y}\to \varpi(p_1^!\mathcal O_X)_y.\]
Hence, $T^{\operatorname{naive}}_\mathfrak{p,\xi}$ also induces an isomorphism
\[(p_2^*\omega^k)_y \to \varpi (p_1^!\omega^k)_y.\]
\end{proof}

Now, we are going to use the two finite flat morphisms $p_1$ and $p_2$, and  $T_\fp$, to define a morphism between cohomologies, that we will also denote by $T_\fp$. First we recall that $(p_2^*,p_{2,*})$ is a pair of functors that are adjoint. Furthermore, as $p_1$ is proper (note that $p_1$ is finite), we have that $p_{1,*}=p_{1,!}$, and thus $(p_{1,*},p_1^!)$ is also an adjoint pair.

With these properties and $T_\fp$ we  construct the following morphism
\begin{align*}
R\Gamma(X;\omega^k)=q_*\omega^k \xrightarrow{\text{Adj.}} & q_*p_{2,*}p^*_2\omega^k&\\
=&q_*p_2^*\omega^k\xrightarrow{q_*T_p}q_*p_1^!\omega^k=q_*p_{1,*}p_1^!\omega^k \xrightarrow{\text{Adj.}} q_*\omega^k=R\Gamma(X;\omega^k),
\end{align*}
where Adj. are the counit or unit of the corresponding adjunction and $q$ is the structural morphism. 
\begin{rmk}\begin{itemize}
    \item Over $\operatorname{Spec}(K_\mathfrak{p})$, the morphism $p_1$ is \'etale, thus $p_1^!=p_1^*$ \cite[Theorem 4.8.1]{Lip1960}. In this situation, our cohomological correspondence at degree $0$ coincides with the Hecke correspondence in \cite[Section 3.2]{NiRo2021}. But, since we are working over $\operatorname{Spec}(A_\mathfrak{p})$, we need to use $p_1^!$, as in \cite{BoPi2022}, to obtain the adjoint morphism.  
    \item In the notation of \cite{BoPi2022}, we  have defined a \emph{cohomological correspondence} of $\omega^k$ over $X$.
\end{itemize}
\end{rmk}

\subsection{Serre duality}\label{subsec:Serredual}  We let ${\bf D}=\operatorname{R\underline{Hom}}_{X}(-,q^!A_\fp)$ be the dualizing functor. For each coherent sheaf $\mathcal F$, ${\bf D(\mathcal F)}$ is concentrated in degree $1$. Since $q: X\rightarrow \mathrm{Spec}(A_{\fp})$ is smooth, we have that $q^!A_\fp=\Omega^1_{X/A_\fp}[1]$ (shifted by $1$). Moreover, alos using that $q_*=q_!$ ($q$ is proper)  we obtain the following Serre duality
\begin{align*}
\operatorname{H}^0(X,\omega^k)^*=\operatorname{R\underline{Hom}}_{\operatorname{Spec}(A_\fp)}(q_*\omega^k,A_\fp)&={\bf D}(\omega^k)\\ &=\operatorname{R\underline{Hom}}_{\operatorname{Spec}(A_\fp)}(A_\fp,q_*(q^!A_\fp\otimes(\omega^k)^*))\\ 
&=\operatorname{H}^0(A_\fp,q_*(\omega^{-k}\otimes\Omega^1_{X/A_\fp})[1])\\
&=\operatorname{H}^1(X,\omega^{-k}\otimes\Omega^1_{X/A_\fp}).
\end{align*}
In order to get a more concrete description of the module $H^1(X,\omega^{-k}\otimes\Omega^1_{X/A_\fp})$, we use the Kodaira--Spencer map developed by \cite{Ge1990,Tag1995}. Let us recall the  ingredients of that construction in the Drinfeld modular curve case.

We start with a general discussion about the Kodaira--Spencer map, following \cite{Ha2020dual}. Recall for a given Drinfeld module $E$ over an affine scheme $S$, there is a dual Drinfeld module that we denote by $E^D$. Furthermore, there is a version of de Rham cohomology that we denote by ${\rm DR}(E,\mathbb{G}_a)$, and the Hodge filtration \cite[eq. (2.8)]{Ha2020dual}:
\begin{equation}\label{eq:hod}
    0\to {\rm Lie}(E)^{-1}\to {\rm DR}(E,\mathbb{G}_a)\to {\rm Lie}(E^D)\to 0
\end{equation}
functorial on $E$.

The Kodaira--Spencer map for $E$ over an $A_\mathfrak p$-scheme $S=\operatorname{Spec}(B)$ is defined by 
\begin{equation}\label{eq:KS}
    {\rm KS} \colon {\rm Der}_{A_\mathfrak p}(B) \to {\rm Hom}_B({\rm Lie}(E)^{-1},{\rm Lie}E^D), 
\end{equation}
with dual map given by
\[  {\rm KS}^\vee \colon {\rm Lie}(E)^{-1}\otimes_B {\rm Lie}(E^D)^{-1}  \to  \Omega^1_{S/{A_\mathfrak p}}.\]
We  apply this to the universal Drinfeld module $\mathcal E$ and $S=Y$. This gives us the relation
\[\omega\otimes \omega_D \xrightarrow[]{\sim} \Omega^1_{X/A_\mathfrak p}(2D) \]
where $D$ is the boundary divisor  of the noncompactified curve $Y$ in $X$ and $\omega_D$ is described as ${\rm Lie}(\mathcal E^D)^{-1}$ over $Y$ (see proof of \cite[Lemma 4.1]{Ha2020dual}). Finally, this allows us to consider ${\bf D}(T^{\text{naive}}_\mathfrak p)$ as a morphism $p_1^*(\omega^{1-k}\otimes \omega_D(-2D)) \to p_2^!(\omega^{1-k}\otimes \omega_D(-2D))$.

On the other hand, recall that we have an universal isogeny $\pi\colon p_1^* \mathcal E\to p_2^*\mathcal E$. Applying the dual and using the fact that duality commutes with base change, we have the following isogeny $\pi^D \colon p_2^* \mathcal E^D\to p_1^*\mathcal E^D$. The tangent space map induces morphisms $\pi^D_{k} \colon p_1^*\omega_D^{k} \to p_2^* \omega_D^k$ for every ${k}\in \mathbb Z$. Observe that $\pi_1^D={\rm Lie}(\pi^D)^{-1}$.  Using the functoriality of Equation \eqref{eq:hod}, applied to $\pi$, we get the following result.

\begin{prop}\label{prop:KS&Serre}
Let $\omega_{X_0(\mathfrak p)/A_\mathfrak p}$ be the dualizing module of $X_0(\mathfrak p)$ over $A_\mathfrak p$. The following diagram is commutative
 \[\begin{tikzcd}
        & \omega_{X_0(\mathfrak p)/A_\mathfrak p}(D_0(\mathfrak p))& \\
        p_2^*\Omega^1_{X/A_\mathfrak p}(D)\arrow[ru, "{\rm tr}_{p_2}"]& & p_1^*\Omega^1_{X/A_\mathfrak p}(D) \arrow[lu, "{\rm tr}_{p_1}"']\\ 
       p_2^*\omega\otimes p_2^*\omega_D\arrow[u, "p_2^*{\rm KS}^\vee"] \arrow[rr,"\pi_1\otimes (\pi^D_1)^{-1}"] & & p_1^*\omega \otimes p_1^*\omega_D.\arrow[u, "p_1^*{\rm KS}^\vee"']
    \end{tikzcd}\]
\end{prop}
\begin{proof} Like in \cite[Lemma 3.7]{BoPi2022} the key is to show that the following is commutative
   \[\begin{tikzcd}
         \omega_{X_0(\mathfrak p)/A_\mathfrak p}(D_0(\mathfrak p))\arrow[r,"1"] & \omega_{X_0(\mathfrak p)/A_\mathfrak p}(D_0(\mathfrak p)) \\
        p_2^*\Omega^1_{X/A_\mathfrak p}(D)\arrow[u, "{\rm tr}_{p_2}"] & p_1^*\Omega^1_{X/A_\mathfrak p}(D) \arrow[u, "{\rm tr}_{p_1}"']\\ 
       p_2^*\omega\otimes p_2^*\omega_D\arrow[u, "p_2^*{\rm KS}^\vee"] \arrow[r,"\pi_1\otimes (\pi^D_1)^{-1}"]  & p_1^*\omega \otimes p_1^*\omega_D.\arrow[u, "p_1^*{\rm KS}^\vee"']
    \end{tikzcd}\]

Over $X_0(\mathfrak p)$  we have  the universal isogeny $\pi: p_1^* \mathcal E \rightarrow  p_2^*\mathcal E$ and by duality and functoriality  we get a morphism between the two exact sequences \eqref{eq:hod}:  Thus when applying \eqref{eq:KS} to $p_1^*\mathcal E$ and $p_2^*\mathcal E$ over $Y_0(\mathfrak p)=\operatorname{Spec}(B')$, we obtain
\[\begin{tikzcd}
        & \omega_{X_0(\mathfrak p)/A_\mathfrak p}(D_0(\mathfrak p))& \\
       p_2^*\omega\otimes p_2^*\omega_D\arrow[ur, ""] \arrow[rr,"\pi_1\otimes (\pi^D_1)^{-1}"] & & p_1^*\omega \otimes p_1^*\omega_D.\arrow[ul, ""']
    \end{tikzcd}\]
Finally, we note that the Kodaira-Spencer map  $p_i^*\omega\otimes p_i^*\omega_D\to\omega_{X_0(\mathfrak p)/A_\mathfrak p}(D_0(\mathfrak p))$, after identifying  $\omega_{X_0(\mathfrak p)/A_\mathfrak p}\cong p_i^!\mathcal{O}_X\otimes p_i^*\Omega^1_{X/A_\mathfrak p}$, is equal to ${\rm tr}_{p_i}\otimes p_i^*{\rm KS}^\vee$, for $i=1,2$.
\end{proof}

\begin{prop}\label{prop:DualTp} We have ${\bf D}(T^{\text{naive}}_\mathfrak p)$ is equal to $\pi_{1-k}^{-1}\otimes {\rm tr}_{p_2}\otimes \pi^D_1 $.
\end{prop}
\begin{proof}
     The dual of the morphism $T_\mathfrak p^{\naive}:p_2^*\omega^k \to p_1^*\omega^k\to p_1^!\omega^k$ is
    \[{\bf D}(T_\mathfrak p^{\naive})\colon p_1^*(\omega^{-k}\otimes \Omega^1_{X/A_\fp})\to p_1^!(\omega^{-k}\otimes \Omega^1_{X/A_\fp}) \to p_2^!(\omega^{-k}\otimes \Omega^1_{X/A_\fp}).\]
    From Serre duality, we know that the first morphism is identified with ${1\otimes{\rm tr}_{p_1}}: p_1^{\ast}\omega^{-k}\rightarrow p_1^{!}\omega^{-k}=p_1^*\omega^{-k}\otimes p_1^!\mathcal O_X$. From the projection formula we have that the second morphism  can be identified with the rational map
    \[\pi_{-k}^{-1}\otimes 1 \colon p_1^*\omega^{-k}\otimes\omega_{X_0(\mathfrak p)/A_\mathfrak p}\to  p_2^*\omega^{-k}\otimes\omega_{X_0(\mathfrak p)/A_\mathfrak p}.\]
    On the other hand, using again the projection formula and Kodaira--Spencer we have 
    \[\omega_{X_0(\mathfrak p)/A_\mathfrak p}\cong  p_1^!\mathcal{O}_X\otimes p_1^*\Omega_{X/A_\mathfrak p}^1\cong p_1^!\mathcal{O}_X\otimes p_1^*\omega\otimes  p_1^*\omega_D (-2D)\]
    and 
    \[\omega_{X_0(\mathfrak p)/A_\mathfrak p}\cong p_2^!\mathcal{O}_X\otimes p_2^*\Omega^1_{X/A_\mathfrak p}\cong p_2^!\mathcal{O}_X\otimes p_2^*\omega\otimes p_2^*\omega_D (-2D).\]
    Now looking at Proposition \ref{prop:KS&Serre} and using the identification above, we have that the identity morphism $\omega_{X_0(\mathfrak p)/A_\mathfrak p}\rightarrow \omega_{X_0(\mathfrak p)/A_\mathfrak p}$ corresponds to 
    \[{\rm tr}_{p_2}\circ{\rm tr}_{p_1}^{-1}\otimes \pi_1^{-1}\otimes \pi^D_1: p_1^!\mathcal{O}_X\otimes p_1^*\omega\otimes  p_1^*\omega_D (-2D)\rightarrow  p_2^!\mathcal{O}_X\otimes p_2^*\omega\otimes p_2^*\omega_D (-2D)\]
    We finally get that 
    \[{\bf D}(T_\mathfrak p^{\naive})=\pi_{-k}^{-1}\otimes {\rm tr}_{p_2}\circ{\rm tr}_{p_1}^{-1}\circ {\rm tr}_{p_1}\otimes \pi_1^{-1}\otimes \pi^D_1 =\pi_{1-k}^{-1}\otimes {\rm tr}_{p_2}\otimes \pi^D_1.\]
\end{proof}
\begin{rmk}
In contrast to what happens in the elliptic curve case, where there is a relation between $\Omega^1_{X/A_\fp}$ and $\omega$, in the Drinfeld modular curve we also need to consider its dual Drinfeld module. This is due to the lack of self-duality of Drinfeld modules, see \cite[Remark 2.20]{Ha2020dual}.
\end{rmk}

\section{The correspondence on the special fiber}\label{subsec:corrmodp} We need the following local analysis of $T_\mathfrak p$ on the supersingular locus of $X_1$, that we denote by ${\rm SS}$. Roughly speaking, this means that the correspondence $T_\mathfrak p$ reduces the order of the poles on the supersingular locus for $k \geq 2$, and for $k \leq 0$ it decreases the order of vanishing.

Recall that the Hasse invariant $\mathrm{Ha}$ is $\operatorname{Lie}(V_{\mathfrak p,\cE})^{\vee}\colon \omega\to  \omega^{q^d}$ and it is non-vanishing at a Drinfeld module $E$ if and only if $E$ is ordinary if and only if $\operatorname{ker} V$ is \'etale over $S$ \cite[Proposition 2.14]{Shas2007}. When we are in the case of the universal Drinfeld module $\mathcal E_1$ over $Y_1$, we will also consider the Hasse invariant as an element in $H^0(Y_1,\omega^{q^d-1})= {\rm Hom}(\mathcal O_{Y_1},\omega^{q^d-1})={\rm Hom}(\omega,
\omega^{q^d})$. This can be further extended to $X_1$, and its vanishing locus is reduced \cite[Corollary 2.16]{NiRo2021}.

\begin{prop} For all $k\geq 2$ and $n\in \mathbb{Z}$, the correspondence $T_\mathfrak p$ over the special fiber induces a map:
    \[p_2^*(\omega^k((nq^d+k-2)\rm{SS})\to p_1^!(\omega^k(n{\rm SS})).\]
  For all $k\leq 0$ and $n\in \mathbb{Z}$, the correspondence $T_\mathfrak p$ over the special fiber induces a map:
  \[p_2^*(\omega^k(-n{\rm SS}))\to p_1^!(\omega^k((-nq^d+k){\rm SS})).\]
\end{prop}
\begin{proof}
    We first prove the case $k\geq 2$. In this case, the support of the correspondence $T_\mathfrak p$ is $X_0(\mathfrak p)_1^V$ because of equation (\ref{e:T_p in F}). Now, the map $p_1^V$ is totally ramified of degree $q^d$ and $p_2^V$ is an isomorphism. It follows that we have the following equality of divisors: \[(p_1^V)^*({\rm SS})=q^d(p_2^V)^*({\rm SS}).\]  
   Thus, the map $(p_2^V)^*\omega^2 \to (p_1^V)^!\omega^2$, the cohomological correspondence $T_\mathfrak p$ restricted to $X_0(\mathfrak p)_1^V$, induces a morphism 
   \[(p_2^V)^*\omega^2(nq^d{\rm SS})\to (p_1^V)^!\omega^2(n{\rm SS}).\] This proves the claim for $k=2$. For $k\geq 3$, we first remark that the correspondence is the tensor product of the map $(p_2^V)^*\omega^2\to(p_1^V)^!\omega^2$ and the map $(p_2^V)^*\omega^{k-2}\to (p_1^V)^*\omega^{k-2}$. But $(p_1^V)^*\omega\cong ((p_2^V)^*\omega)^{q^d}$ and $\mathrm{Ha}$ is defined in terms of the universal isogeny, then we have an isomorphism 
    \[(p_2^V)^*(\omega({\rm SS}))\xrightarrow{(p_2^V)^*\operatorname{Ha}}((p_2^V)^*\omega)^{q^d}\cong  (p_1^V)^*\omega,\] 
    and therefore, for $k\geq 2$, we have that the map $(p_2^V)^*(\omega^{k-2}((k-2){\rm SS}))\to (p_1^V)^*\omega^{k-2}$ induces an isomorphism. We deduce  that there is a map $p_2^*(\omega^k((nq^d+k-2){\rm SS}))\to p_1^!(\omega^k(n{\rm SS}))$.

    Now we prove the case $k\leq 0$. The correspondence is supported on $X_0(\mathfrak p)_1^F$. Indeed, again this follow from the proof of Proposition \ref{prop:Tpint}, more precisely from equation \eqref{eq:TpFrob}, $p_1^F$ is an isomorphism and the definition of the normalization (Section \ref{subsec:corrTp}). Now, the map $p_2^F$ is totally ramified of degree $q^d$ and the map $p_1^F$ is an isomorphism. It follows that we have an equality of divisors
\[(p_2^F)^*({\rm SS})=q^d(p_1^F)^*({\rm SS}).\] We deduce that  the  \[(p_2^F)^*\mathcal{O}_X\to (p_1^F)^!\mathcal{O}_X\] induces a morphism
    $(p_2^F)^*(\mathcal{O}_X(-n{\rm SS}))\to (p_1^!)(\mathcal{O}_X(-nq^d{\rm SS})).$
    This proves the case $k=0$. For $k\leq -1$, we remark that the correspondence is the tensor product of the map $(p_2^F)^*\mathcal{O}_X\to (p_1^F)^!\mathcal{O}_X$ and a map $(p_2^F)^*\omega^k\to (p_2^F)^*\omega^k$ that is deduced from the differential of the Verschiebung map $E^{(\mathfrak p)}\to E$.  We observe that $(p_2^F)^*\omega\cong ((p_1^F)^*\omega)^{q^d}$ and there is a natural isomorphism
    \[(p_2^F)^*\omega \cong ((p_1^F)^*\omega)^{q^d} \xrightarrow{(p_1^F)^*\operatorname{Ha}^{-1}} (p_1^F)^*(\omega({\rm SS})).\]
    Therefore, for all $k\leq 0$, an isomorphism $(p_2^F)^*\omega^k \to (p_1^F)^*(\omega^k(-k{\rm SS})),$ which factors the map
    \[(p_2^F)^*\omega^k \xrightarrow{(p_1^F)^*\operatorname{Ha}^k}(p_1^F)^*\omega^k.\]
    We deduce that there is a map $p_2^*(\omega^k(-n{\rm SS}))\to p_1^!(\omega^k((q^d+k){\rm SS})).$
\end{proof}

Now we use the formalism of six operations that naturally appears in condensed mathematics. This is done by looking at the derived category of quasi-coherent  $\mathcal O_{X^{\rm ord}_1}$-modules $D(\mathcal O_{X_1^{\rm ord}})$ as the discrete object of solid $\mathcal O_X$-modules $D(\mathcal O_{X^{\rm ord}_1,\blacksquare})$. In the latter, we have a six-functor formalism \cite{ClSc2019}.  In particular there is $q_{1,!}$, where $q_1\colon X_1^{\ord}\to \operatorname{Spec}(k(\mathfrak p))$ is the structure map of the ordinary locus. We put $\operatorname{H}^1_c(X^{\ord}_1,\omega^{k})\coloneqq \operatorname{H}^1(q_{1,!}\omega^{k})$ where $q_{1,!}\omega^{k}$ is viewed as a complex of $k(\mathfrak p)$-modules. Remark that via \cite[Prop. 2.3.5]{BoPi2023} this $k(\mathfrak p)$-vector space is profinite. Moreover, this definition is exactly the same as used in \cite[\S4.2.2]{BoPi2022}, where the authors used a more classical approach, without using condensed mathematics. In particular, we have $\operatorname{H}^1_c(X^{\ord}_1,\omega^k)=\lim_n \operatorname{H}^1_c(X_1,\omega^k(-n\rm{SS}))$.

Recall that  a continuous endomorphism $T\colon M\to M$ over a finite Artinian ring $R$ of a profinite $R$-module $M$ is \textit{locally finite} if $M$ has a basis of neighborhoods of zero consisting of submodules $\{N_i\}$ such that $T(N_i) \subseteq N_i$. This notion was introduced in Definition 2.9 of \cite{BoPi2022}. Now, as in Section 4 of [\emph{loc.\,cit.}], we have:
\begin{cor} For every $i\geq 0$, we have the following
    \begin{enumerate}
        \item if $k\geq 2$, $T_\mathfrak p$ is locally finite on $\operatorname{H}^i(X_1^{\ord},\omega^k)$,
        \item if $k\leq 0$, $T_\mathfrak p$ is locally finite on $\operatorname{H}^i_c(X_1^{\ord},\omega^k)$.
    \end{enumerate}
\end{cor}
Thanks to this corollary, we can define the projector $e(T_\mathfrak p)$ as in [\emph{loc.\,cit.}]. We have the following classicality results for the ordinary projection of modular forms modulo $\varpi$:

\begin{cor}\label{cor:descent} For every $i\geq 1$ we have that:
    \begin{enumerate}
        \item $e(T_\mathfrak p)\operatorname{H}^i(X^{\ord}_1,\omega^k)=\begin{cases}
            e(T_\mathfrak p)\operatorname{H}^i(X_1,\omega^k) &\text{ if }k\geq 3, \\
            e(T_\mathfrak p)\operatorname{H}^i(X_1,\omega^2({\rm SS})) &\text{ if }k=2.
        \end{cases}$
        \item $e(T_\mathfrak p)\operatorname{H}^i_c(X^{\ord}_1,\omega^k)=\begin{cases}
            e(T_\mathfrak p)\operatorname{H}^i(X_1,\omega^k) &\text{ if }k\leq -1, \\
            e(T_\mathfrak p)\operatorname{H}^i(X_1,\mathcal{O}_X(-{\rm SS})) &\text{ if }k=0.
        \end{cases}$
    \end{enumerate}
\end{cor}

A similar result holds for generic characteristic:

\begin{cor}\label{cor:classicTp}
    The corestriction map induces the following isomorphisms
    \begin{enumerate}
        \item For $k\geq 3$ we have an isomorphism $e(T_\mathfrak p)\operatorname{H}^0(\mathfrak X^{\ord},\omega^k)\cong e(T_\mathfrak p)H^0(X,\omega^k)$
        \item For each $k\leq -1$ we have an isomorphism   \[e(T_\mathfrak p)\operatorname{H}^1_c(\mathfrak X^{\ord},\omega^{2-k}(-2D))\cong e(T_\mathfrak p)\operatorname{H}^1(X,\omega^{1-k}\otimes\omega_D(-2D)). \]
    \end{enumerate}
\end{cor}
\begin{proof}
\begin{enumerate}
        \item By Corollary \ref{cor:descent}, we know the results  modulo $\mathfrak p$. To deduce the result over $A_{\mathfrak p}$, we need to prove the surjectivity of reduction modulo $\mathfrak p$. This can be proven using the long exact sequence associated with the multiplication by $\varpi$ on $\omega^k$ (either on $X^{\ord}$ or $X$) and using the vanishing of $ \operatorname{H}^1$. For  $X^{\ord}$ this follows from it being affine, for $X$ one uses Serre duality, see Section \ref{subsec:Serredual}. Finally, by \cite[Th\'eor\`eme 3.15]{NiRo2021} $U_{\mathfrak p}$ and $T_{\mathfrak p}$ are congruent modulo $\varpi$ if $k>2$.
        One can use the the same proof as \cite[Theorem 4.3, 4.6 (3)]{BrascaRosso} using the results of \cite{NiRo2021}. 
        \item If follows from the previous point by Serre duality and the isomorphism $\omega\cong\omega_D$ in $\mathfrak X^{\rm ord}$.
       \end{enumerate}
 
\end{proof}

\section{\texorpdfstring{$\mathfrak p$-adic theory}{p-adic theory}}
The goal of this section is to give a construction of families of $\omega^k$, $k\in\mathbb Z$. For this, we review some of the basics around Taguchi duality. This allows us to consider the Igusa tower, which is the key to the construction of these families.


\subsection{Canonical subgroups and Hodge--Tate--Taguchi maps}

We review some basics on canonical subgroups in the context of Drinfeld modules of rank two and the Taguchi duality. For this, we mainly follow the treatment in Section 3 of \cite{Ha2022cpt}.


Let $Y^{\ord}_n$ be the ordinary locus of $Y_n= Y\times A/\fp^n$. We recall first that there exists a finite locally free closed $A$-submodule $H^{\can}_{n}$ over $X_n$ of $E[\mathfrak p]$ extending the canonical subgroup over $Y^{\ord}_n$ of the universal ordinary Drinfeld module of rank two $\mathcal{E}^{\ord}[\fp^n]$ modulo $\mathfrak p^m$ such that its Cartier--Taguchi dual is \'etale locally isomorphic to $(A/\fp^n)$ \cite[Lemma 4.9]{Ha2020dual}.

\begin{rmk}\label{rmk:tag-dual}
\begin{itemize}
    \item  In the notations of \cite[Lemma 4.9]{Ha2020dual}, the group is $H^{\rm can}_{n}$ is $\mathcal C_{n,m}$ with $n=m$, where the first $n$ express the level and the other the one indicates that it modulo $\mathfrak p^m$. 

\item In this positive characteristic context, we use the Cartier--Taguchi dual $(H^{\can}_{n})^D=\operatorname{Hom}(H^{\can}_{n},C)$ of $H^{\can}_{n,m}$,  where $C$ denotes the Carlitz(–Hayes) module. With this notion, one actually has that $(H_{n}^{\can})^D$ is \'etale, contrary to what happens if we use the classical Cartier dual, were we would get that its dual is isomorphic to a non-reduced algebra, a non-\'etale group.

\end{itemize}
\end{rmk}

The theory of Cartier--Taguchi duals leads to the Hodge--Tate--Taguchi map
\[\operatorname{HTT}\colon  (H^{\can}_{n})^D\to \omega_{H^{\can}_{n}}.\]
where $\omega_{H^{\can}_{n}}$ is the differential attached to $H^{\can}_{n}$. It sends $s\colon H^{\can}_{n}\to C$, to $s^*dz$, where $dz$ is the canonical generator of the co-lie algebra of $C$. Furthermore, modulo $\fp^n$, this induces an isomorphism \cite[Proposition 3.8 and 4.15]{Ha2020dual} 
\[\operatorname{HTT}\otimes 1\colon (H_{n}^{\can})^D \otimes_{A_\fp/\fp^n}\mathcal O_{X^{\ord}_n}\to \omega/\fp^n.\]
Passing to the limit over $n$, we have the following isomorphism
\[\operatorname{HTT}\otimes 1\colon \lim(H^{\can}_{n})^D \otimes\mathcal O_{\mathfrak X^{\ord}}\to \omega.\]

\subsection{Universal weights} 
We note that $\lim(H_{n}^{\can})^D$ is a pro-\'etale sheaf, pro-\'etale locally isomorphic to $A_\fp$. We denote it by $T_\fp((H^{\can})^D)$.

We define the Igusa tower as the torsor of ``frames'' of $T_\fp((H^{\can})^D)$:
\[\operatorname{Ig}=\operatorname{Isom}_{\mathfrak X^{\ord}}(A_\fp,T_\fp((H^{\can})^D))\to \mathfrak X^{\ord}.\]

It is an $A_\fp^\times$-torsor and via the Hodge--Tate--Taguchi map we can get
\[\omega\cong T_\fp((H^{\can})^D)\otimes\mathcal O_{\mathfrak X^{\ord}}\cong(\mathcal{O}_{\operatorname{Ig}}\widehat\otimes A_\fp )^{A_\fp^\times},\]
where the action is taken diagonally and the action on $A_\fp$ is by multiplication.

This relation allows us to construct families of $\omega^k$, $k\in \mathbb{Z}$. Indeed, let $\kappa\colon A_\fp^\times \to R^\times$ be a continuous homomorphism, where $R$ is a linearly topologized ring over $A_\fp$. Then 
\[\omega^\kappa\coloneqq(\mathcal{O}_{\operatorname{Ig}}\widehat \otimes R)^{A_\fp^\times},\]
where the action is taken diagonally, with the action  on $R$ given by $\kappa$ and the tensor product is the completed one \cite[\href{https://stacks.math.columbia.edu/tag/0AMQ}{Tag 0AMQ}]{stacks-project}.

This construction satisfies that $\omega^\kappa$ is an invertible sheaf over $\mathcal{O}_{\mathfrak X^{\ord}}\widehat\otimes R$ and that it is functorial the following sense: let $f\colon R\to R'$ and $\kappa\colon A_\fp^\times \to R^\times$. If $\kappa'=f\circ\kappa\colon A_\fp^\times \to R'^\times $, then
\[\omega^{\kappa'} \cong \omega^\kappa \otimes_{R,f}R'.\]
These families have a universal one. Let $\Lambda\coloneqq A_\fp\llbracket A_\fp^\times \rrbracket$ and $\kappa^{un}\colon A_\fp^\times \to \Lambda^\times$, the universal character. The \emph{universal family} is defined by $\omega^{\kappa^{un}}$.

We note that if $\kappa\colon A_\fp^\times\to R^\times$, then there exists a unique $A_\mathfrak p$-algebra $f_{\kappa}\colon \Lambda \to R$ such that $\kappa=f\circ\kappa^{un}$. Thus, by the functoriality, we have that
\[\omega^\kappa\cong \omega^{\kappa^{un}} \otimes_{\Lambda,f_\kappa}R.\]
In particular, if $k:A_\fp^\times \to A_\fp^\times$ is the homomorphism sending $x$ to $x^k$, where $k\in \mathbb{Z}$, then \begin{equation}\label{eq:sp k}
\omega^k\cong \omega^{\kappa^{un}} \otimes_{\Lambda,f_k}A_\fp.    
\end{equation}  In words, the specialization  of $\omega^{\kappa^{un}}$ via $f_k$ is $\omega^k$. We thus have constructed families of $\omega^k$, $k\in \mathbb{Z}$. 

\subsection{Frobenius and \texorpdfstring{$U_\fp$}{Up}} \label{subsec:FU}
We have a map $F:\mathfrak X^{\rm ord} \to \mathfrak X^{\rm ord}$, sending $E\mapsto E/H^{\can}_1$, and the level structure on $E/H_1^{\rm can}$ induced by the isogeny $E\to E/H_1^{\rm can}$. We can extend this to $\operatorname{Ig}$ as follows. 
By definition, $\operatorname{Ig}$ parametrizes (in a dense open subset) ordinary Drinfeld modules with level structure and a trivialization of $T_\fp((H^{\can})^D))$. This allows us to define $F\colon \operatorname{Ig}\to \operatorname{Ig}$, as follows
\[(E, \psi:A_{\mathfrak p}\xrightarrow{\sim} T_\fp((H^{\can})^D))\mapsto(E/H^{\can}_1, \psi':A_{\mathfrak p}\xrightarrow{\sim} T_\fp((H^{\can}/H^{\can}_1)^D)
),\]
where the level structure of $E/H^{\rm can}_1$ is the one induced from $E$ and the natural projection, and $\psi' \colon A_{\fp} \xrightarrow{\varpi\psi} \varpi T_\fp((H^{\can})^D)  \cong T_\fp((H^{\can}/H^{\can}_1)^D)$. 
\begin{prop} The section morphism $F^* \mathcal{O}_{\rm Ig}\to \mathcal{O}_{\rm Ig}$ of $F\colon \operatorname{Ig}\to \operatorname{Ig}$ descends to  a morphism $F^*\omega^{\kappa^{un}}\to \omega^{\kappa^{un}}$.
\end{prop}
By abuse of notation, we will still denote the morphism $F^*\omega^{\kappa^{un}}\to \omega^{\kappa^{un}}$  by $F$.

We also consider a variant $F'\colon \mathfrak X^{\rm ord}\to \mathfrak X^{\rm ord}$ which is defined in the same way as $F$ except that we endow $E/H_1^{\can}$ with the level structure via the dual isogeny, i.e., $F'=\langle\varpi\rangle^{-1}\circ F$, where $\langle \varpi \rangle $ is the operator on ${\mathfrak X}$ that acts on the prime-to-$\mathfrak p$ level structure as multiplication by $\varpi$. As before, we extend this map to $\operatorname{Ig}$.  

\begin{prop} The trace morphism $\operatorname{Tr}_{F'}: F'_*\mathcal{O}_{\rm Ig}\to \mathcal{O}_{\rm Ig}$ of $F'\colon \operatorname{Ig}\to \operatorname{Ig}$ satisfies  $\operatorname{Tr}_{F'}(F'_*\mathcal{O}_{\rm Ig})\subset\varpi \mathcal{O}_{\rm Ig}$ and $\varpi^{-1}\operatorname{Tr}_{F'}$ descends to  a morphism $U_\fp \colon F'_*\omega^{\kappa^{un}}\to \omega^{\kappa^{un}}$.
\end{prop}

\subsection{Frobenius, \texorpdfstring{$U_\fp$}{Up} and \texorpdfstring{$T_\fp$}{Tp}} \label{subsec:FUT}
Let us study the specialization to an integer $k$ of the previous morphism. By abuse of notation we still denote them by $F$ and $U_\mathfrak p$. 

We know that on $\mathfrak X_0(\mathfrak p)^{\ord,F}$, the map $p_1$ is an isomorphism and the map $p_2$ identifies with the map $F\colon \mathfrak X^{\rm ord}\to\mathfrak X^{\rm ord}$. On $\mathfrak X_0(\mathfrak p)^{\ord,V}$  the map $p_2$ is an isomorphism and the map $p_1$ identifies with the map $F'\colon \mathfrak X^{\rm ord}\to\mathfrak X^{\rm ord}$. Then we can identify $F$ and $U_\mathfrak p$ as correspondences $p_2^*\omega^k \to p_1^!\omega^k$ supported respectively on $\mathfrak X_0(\mathfrak p)^{\ord,F}$  and $\mathfrak X_0(\mathfrak p)^{\ord,V}$. With this, can compare $F$ and $U_\mathfrak p$ with $T_\mathfrak p$.

\begin{prop}\label{pro:FUT} We have that $F$ is, up to a non-zero constant,  equal to the projection on the component $\mathfrak X_0(\mathfrak p)^{\ord,F}$ of $T_\mathfrak p$. Similarly,  $U_\mathfrak p$ is up to a non-zero constant  equal to the projection on the component $\mathfrak X_0(\mathfrak p)^{\ord,V}$ of $T_\mathfrak p$.
\end{prop}
\begin{proof}
   For $U_{\mathfrak p}$, note that by \cite[Th\'eor\`eme 3.15]{NiRo2021} $U_{\mathfrak p}$ and $T^{\naive}_{\mathfrak p}$ divided by $\varpi$  are congruent modulo $\mathfrak p$ if $k>2$, this is by because, by the normalization of [{\it loc. cit.}], the action of the matrix  $\left( \begin{array}{cc}
    \varpi & 0 \\
    0 & 1
\end{array}\right)$ is $0$ modulo $\mathfrak p$.

   For negative $k$, we need to show that the cohomological correspondence \[T_\mathfrak p\colon p_2^*(\omega^{1-k}\otimes\omega_D) \to p_1^!(\omega^{1-k}\otimes\omega_D)\] (under the isomorphism $\omega\cong\omega_D$ which holds on the ordinary locus) vanishes on $\mathfrak X_0(\mathfrak p)^{\ord,V}$. 
   By the discussion in Proposition \ref{prop:Tpint}, $T_{\mathfrak p}^{\naive}$ is already integral over $\mathfrak X_0(\mathfrak p)^{\ord,V}$, so multiplication by $\varpi^{-k}$ is $0$ modulo $\mathfrak{p}$.
\end{proof}
\begin{rmk}
    Compare this with the classical expression of $T_p$ for elliptic modular forms (see \cite[Lemma 4.13]{BoPi2022}):
    \[
\frac{1}{p}\sum_{j=0}^{p-1} |_k \left( \begin{array}{cc}
    1 & j \\
    0 & p
\end{array}\right) + p^{k-1}|_k \left( \begin{array}{cc}
    p & 0 \\
    0 & 1
\end{array}\right).
    \]
    If $k >0$, then modulo $p$ the last term vanishes and what is left is $U_p$. If $k<0$, then it needs to be renormalized by $p^{1-k}$, and modulo $p$ the first $p$ terms vanish, and we are left with the action of $\left( \begin{array}{cc}
    p & 0 \\
    0 & 1
\end{array}\right)$ that on $q$-expansion is 
\[
\sum_{n=1}^\infty a_n q^n \mapsto \sum_{n=1}^\infty a_n q^{pn},
\]
which modulo $p$ is exactly the Frobenius. 
\end{rmk}
We can now improve Corollary \ref{cor:classicTp} using the previous proposition:
\begin{cor}\label{cor:classic}
    The corestriction map induces the following isomorphisms
    \begin{enumerate}
        \item For $k\geq 3$ we have an isomorphism $e(U_{\mathfrak p})\operatorname{H}^0(\mathfrak X^{\ord},\omega^k)\cong e(T_\mathfrak p)\operatorname{H}^0(X,\omega^k)$
        \item For each $k\leq -1$ we have an isomorphism   \[e(F)\operatorname{H}^1_c(\mathfrak X^{\ord},\omega^{2-k}(-2D))\cong e(T_\mathfrak p)\operatorname{H}^1(X,\omega^{1-k}\otimes\omega_D(-2D)). \]
    \end{enumerate}
\end{cor}


\section{Higher Hida Theory} \label{sec:HHidaTh}

\subsection{On the Iwasawa algebra}
Let $G$ be a profinite group, and suppose that it is isomorphic to $\lim_m G_m$, where the $G_m$'s are finite groups. Let us write $A_\mathfrak p\llbracket G\rrbracket= \lim_m (A_{\mathfrak p}/\mathfrak p^m )[G_m]$.
 This algebra is called the Iwasawa algebra of $G$. In the special case when $G=A_\mathfrak p^\times$, we denote $A_\mathfrak p\llbracket G\rrbracket$ by $\Lambda$.  Our first goal of this section is to find a suitable description of the profinite structure of $\Lambda$.

\begin{lem}\label{lem:iwalg}
The ring $\Lambda$ is a profinite semilocal ring 
\[
\Lambda \cong \prod_{ \chi \in {(\kappa(\mathfrak p)^\times)}^*} \Lambda_0,  
\]
for $\Lambda_0:=A_{\mathfrak p} \llbracket 1+\mathfrak p  \rrbracket $, the Iwasawa algebra of  $1+\mathfrak p $.
Moreover, we can choose a family of ideals  $I_r$ giving the profinite structure such that the $\Lambda$-action on $I_r/I_{r+1}$ factors via
\[
\prod_{ \chi \in {(\kappa(\mathfrak p)^\times)}^*} \kappa(\mathfrak p)^.
\]
\end{lem}

\begin{proof}
The decomposition of  $\Lambda$ over   ${ \chi \in {(\kappa(\mathfrak p)^\times)}^*}$ is via Fourier analysis (see \textnumero 5 \S 21 of \cite{Brbkalg8}). We have 
\[
A_{\mathfrak p} \llbracket 1+\mathfrak p \rrbracket  \cong \varprojlim_m A_{\mathfrak p}/\mathfrak p^m [(1+\mathfrak p)/(1+\mathfrak p^m)]
\]
which is profinite.
As $1+\mathfrak p$ is a pro-$p$ group in infinitely many generators, \[
A_{\mathfrak p} \llbracket 1+\mathfrak p  \rrbracket  \cong A_{\mathfrak p} \llbracket T_1, T_2, \ldots \rrbracket.
\]
We set $I_0=J_0=(\varpi, T_1, T_2, \ldots)$, $J_n=((\varpi, T_1, T_2, \ldots, T_n)^n, T_{n+1}, \ldots)$. Define $I_2=J_2$.
Suppose we have defined $I_3, \ldots, I_{\alpha(n)}=J_n$. 
Note that a basis of $J_n/J_{n+1}$ is given by
\[
\varpi^{n}, \varpi^{n-1}T_1, \ldots, T_n^n, T_{n+1}, T_{n+1}\varpi, T_{n+1}^2, T_{n+1}T_2 ,\ldots, T_{n+1}T_n^{n-1}.
\]
In short, all the monomials of degree $n$ in $\varpi, T_1, T_2, \ldots, T_n$ and all monomials with $\varpi, T_1, T_2, \ldots, T_n$ and $T_{n+1}$
whose degree is less than $n+1$ and divisible by $T_{n+1}$.
We define $I_{\alpha(n)+1}=((\varpi, T_1, T_2, \ldots, T_n)^{n+1}, T_{n+1}, \ldots)$. The action of $\Lambda$ on $I_{\alpha(n)+1}/I_{\alpha(n)}$ is trivial. 
We define  $I_{\alpha(n)+j}$ as generated by:
\begin{itemize}
    \item $(\varpi, T_1, T_2, \ldots, T_n)^{n+1}$;
    \item all the monomials in $\varpi, T_1, T_2, \ldots, T_n$ and $T_{n+1}$ of degree $j$ which are divisible by $T_{n+1}$;
    \item all the other variables  $ T_{n+2}, \ldots$.
\end{itemize}  For $j=1$ we get back  
$I_{\alpha(n)+1}$ as we only add $T_{n+1}$.
We need to show that \[ \frac{I_{\alpha(n)+j}}{I_{\alpha(n)+j+1}} \] has the trivial $\Lambda$-action. It is generated over $\kappa(\mathfrak p)=A_{\mathfrak p} /\mathfrak p$  by monomials in  $\varpi, T_1, T_2, \ldots, T_n$ and $T_{n+1}$ of exact degree $j$ with at least a single power of $T_{n+1}$. The action of the maximal ideal of $\Lambda$ is zero: the extra variables $T_{n+2}, \ldots $ all act trivially. The action of $\varpi, T_1, T_2, \ldots, T_n$ and $T_{n+1}$ is also trivial, as multiplication by any of these elements on the basis will give a a monomial of  degree $j+1$, still divisible by $T_{n+1}$, which is then in $I_{\alpha(n)+j+1}$.
\end{proof}


\begin{rmk} The ring $\Lambda$ is non-noetherian. For simplicity suppose that $A_\mathfrak p^\times \cong \mathbb F_p\llbracket X\rrbracket^{\times}=\mathbb F_p^\times \times (1+X\mathbb F_p\llbracket X\rrbracket) $. Now in this case, $1+X\mathbb F_p\llbracket X\rrbracket$ has infinitely many $\mathbb Z_p$-topological generators, given by $(1+X)^m$ with $m$ a positive integer coprime with $p$, in formula,
\begin{align*}
    1+X\mathbb F_p\llbracket X\rrbracket &\xrightarrow{\sim}\prod_{m\colon p\nmid m}\mathbb Z_p.\\
    ( 1+X)^m &\mapsto e_m,
\end{align*}
where $e_m$ is canonical $m$-coordinate basis of the Cartesian product. Then $\Lambda = A_{\mathfrak p}[\mathbb F_p^\times]\otimes A_{\mathfrak p}\llbracket X_i \colon i\in \mathbb{N}\rrbracket$. This discussion also show the following: let $m_\Lambda$ be the kernel of the natural projection $\Lambda \to R_1 =\kappa(\mathfrak p)[\kappa(\mathfrak p)^\times]$, then $\Lambda$ is complete in its limit topology, but it is not $m_\Lambda$-adically complete (see \href{https://stacks.math.columbia.edu/tag/05JA}{Tag 05JA}).
\end{rmk}

We give some useful properties of our Iwasawa algebra. For that let us define the following: Let $k \in \mathbb{Z}_p$;  we define the  ideal $P_k$ to be the kernel of the map 
    \[
    \begin{array}{ccccc}
      f_k  : & \Lambda_0    & \longrightarrow  & A_{\mathfrak p} \\
         & 1+\mathfrak p A_{\mathfrak p} \ni [u] & \mapsto & u^k.
    \end{array}
    \]
 We denote by $\mathcal{C}(\mathbb{Z}_p,  A_{\mathfrak p} )$ the space of continuous functions $\mathbb{Z}_p \rightarrow A_{\mathfrak p}$.

\begin{prop} \label{prop:densek} 
 There is an injective morphism $\iota:   \Lambda_0    \longrightarrow    \mathcal{C}(\mathbb{Z}_p,  A_{\mathfrak p} ) $, induced by $1+\mathfrak p A_{\mathfrak p} \ni [u] \mapsto  (s \mapsto u^s)$ and extended  $A_{\mathfrak p} $ linearly.
Moreover, the set of ideals $\left\{ P_k \right\}_{ k \in I}$ is Zariski dense in $\Lambda_0$ as long as $I$ is dense subset of $\Z_p$.
    
\end{prop}

\begin{proof}
   We show that we can embed $\Lambda_0 $ in the dual of $A_{\mathfrak p}$-valued continuous functions on $1+\mathfrak p$.
    Let $G=1+\mathfrak p$. If $x = \varprojlim x_i \in A[G/G_i]$, let $\mu_x$ be the measure such that $\mu_x(gG_i)=a_g$, where $x_i=\sum_{g G/G_i} a_g gG_i$. 
    This is indeed a measure: if  $x_{i+1}=\sum_{g' G/G_{i+1}} a_{g'} {g'}G_{i+1}$ is a lift of $x_i$, then we have $\sum_{g' \equiv g \bmod G_i } a_{g'}=a_g$, so writing $gG_i=\cup_{g' \equiv g \bmod G_i} g'G_{i+1}$ so we get 
    \[
    \sum_{g' \equiv g \bmod G_i} \mu_x(g'G_{i+1}) = \sum_{g' \equiv g \bmod G_i } a_{g'}=a_g = \mu_{x}(gG_i).
    \]
    We extend it to zero outside $1+\mathfrak p $.
  This map is injective  as if $\mu_x(gG_i)=0$ for all $i$'s and $g$, then $x_i=0$ for all $i$'s, and $x =0$.    
    By a result of Kaplansky (see \cite[Theorem 43.3]{Ultrametric calculus}) polynomial functions are dense in $\mathrm{Frac}(A_{\mathfrak p})$-valued continuous functions on $1+\mathfrak p$, so a distribution coming from $\Lambda_0\otimes_{A_{\mathfrak p}}\mathrm{Frac}(A_{\mathfrak p})$  that vanishes on all polynomials must be $0$. But $\otimes_{A_{\mathfrak p}}\mathrm{Frac}(A_{\mathfrak p})$ is flat, so we conclude that $\mu_x=0$ if it integrates all polynomials to $0$.
    
To show that the map $\iota$ is injective, suppose $\iota(x)(s)\equiv 0$. Then in particular it is $0$ when evaluated at a positive integer $k$. But $\mathrm{ev}_k \circ \iota$ is reducing modulo $P_k$ or equivalently $\mu_x(u^k)$, so $\mu_x=0$ by Kaplansky's result, and $x=0$.  For a general subset $I$ of $\mathbb Z_p$, we consider the map
\[
\begin{array}{ccccc}
     \prod_{k\in I} \mathrm{ev}_{k}:  &\mathcal{C}(\mathbb{Z}_p,  A_{\mathfrak p} )    & \longrightarrow  &  \prod_{k \in I} A_{\mathfrak p}  \\
         & f(s) & \mapsto & (f(k))_{k \in I}.
    \end{array}
\]
If $\iota(x)(s)$ vanishes for all $k \in I$, as $I$ is dense in $\Z_p$, then $\iota(x)(s)\equiv 0$ but as $\iota$ is injective, then $x=0$.
\end{proof}

\subsection{\texorpdfstring{The projectors $e(F)$ and $e(U_{\fp})$}{The projectors e(F) and e(Up)}}

We now can consider and define our main $\Lambda$-modules
\begin{center}
    $\operatorname{H}^0(\mathfrak X^{\ord},\omega^{\kappa^{un}})$ and $\operatorname{H}^1_c(\mathfrak X^{\ord},\omega^{\kappa^{un}}).$
\end{center}
Since $\mathfrak X^{\ord}$ is affine  and $\omega^{\kappa^{un}}$ a coherent sheaf, we can use the formalism of \cite{ClSc2019} to define the first cohomology group  of compact support. We are going to recall it. First, we define for every $n$ the cohomology group by $\operatorname{H}^1_c(X^{\ord}_n,\omega^{\kappa^{un}}\otimes (\Lambda/I_n))= \operatorname{H}^1(q_{n,!}\omega^{\kappa^{un}}\otimes (\Lambda/I_n))$, where $q_n\colon X_n^{\ord}\to \operatorname{Spec}(A_\mathfrak p/\mathfrak p^n)$ is the structure map of the ordinary locus. Now, we  define
\[\operatorname{H}^1_c(\mathfrak X^{\ord},\omega^{\kappa^{un}})=\lim \operatorname{H}^1_c(X^{\ord}_{n},\omega^{\kappa^{un}}\otimes (\Lambda/I_n)).\]
We have the following interpolation result:
\begin{prop} \label{prop:spk}
    We have the following identities: 
    \[\operatorname{H}^0(\mathfrak X^{\ord},\omega^{\kappa^{un}}) \otimes_{\Lambda,f_k}A_\fp=\operatorname{H}^0(\mathfrak X^{\ord},\omega^k), \quad k\geq 3.\]
    \[\operatorname{H}^1_c(\mathfrak X^{\ord},\omega^{\kappa^{un}}) \otimes_{\Lambda,f_k}A_\fp=\operatorname{H}^1_c(\mathfrak X^{\ord},\omega^k), \quad k\leq -1.\]
\end{prop}
\begin{proof}
For $\operatorname{H}^0$, it is essentialy \eqref{eq:sp k} combined with the vanishing of $\operatorname{H}^1$ on the ordinary locus.
For $\operatorname{H}^1_c$, we again take the long exact sequence coming from 
\[
0 \rightarrow \mathrm{Ker}(f_k)\rightarrow \Lambda \rightarrow A_\fp \rightarrow 0
\]
combined with the vanishing of $\operatorname{H}^i_c(\mathfrak X^{\ord}, \phantom{A})$ for $i=0,2$.
\end{proof}

Let $M=\lim M/I_nM$ and $T$ a $\Lambda$-endomorphism of $M$, where the $I_n$'s are the ones in Lemma \ref{lem:iwalg}. We say that $T$ is locally finite if for all $n$, the induced endomorphism $M/I_nM$ over the Artinian ring $R_n=\Lambda/I_n$ is locally finite in the sense of \cite[Section 2.2]{BoPi2022} reviewed in Section \ref{subsec:corrmodp}. We note that the definition of locally finite  in this context also comes from [\emph{loc.\,cit.}].

\begin{thm} \label{thm:projandinter}
    There exist locally finite operators $F$ and $U_\mathfrak p$ on $\operatorname{H}^1_c(\mathfrak X^{\ord},\omega^{\kappa^{un}})$ and $\operatorname{H}^0(\mathfrak X^{\ord},\omega^{\kappa^{un}})$, respectively, with projectors $e(F)$  and $e(U_{\mathfrak p})$ , respectively. These are defined via the formulas $e(F)=\lim_n F^{n!}$ and  $e(U_{\mathfrak p})=\lim_n U_{\mathfrak p}^{n!}$. Moreover, for the morphism $f_k$ defined above we have the following identities: 
    \[e(U_{\mathfrak p})\operatorname{H}^0(\mathfrak X^{\ord},\omega^{\kappa^{un}}) \otimes_{\Lambda,f_k}A_\fp=e(U_{\mathfrak p})\operatorname{H}^0(\mathfrak X^{\ord},\omega^k), \quad k\geq 3.\]
    \[e(F)\operatorname{H}^1_c(\mathfrak X^{\ord},\omega^{\kappa^{un}}) \otimes_{\Lambda,f_k}A_\fp=e(F)\operatorname{H}^1_c(\mathfrak X^{\ord},\omega^k), \quad k\leq -1.\]
\end{thm}

\begin{proof}
The result for $\operatorname{H}^0$ is essentially in \cite{NiRo2021} so we prove only the result for $\operatorname{H}^1_c$.  We follow very closely the arguments in \cite[Section 4.2]{BoPi2022}. As there, we focus on $F$ and aim to construct an action $F$ on $\operatorname{H}^1_c(X^{\ord}_{n},\omega^{\kappa^{un}}\otimes (\Lambda/I_n))$, compatible for all $n$. Recall that we write $X_0(\mathfrak p)_n \to {\rm Spec}(A_\mathfrak p/\mathfrak p^n)$ for the reduction modulo $\mathfrak p^n$ of $X_0(\mathfrak p)$, and $X_0(\mathfrak p)^{\ord}_n$ for its ordinary locus.

First we construct a cohomological correspondence on  $X_0(\mathfrak p)^{\ord}_n$: 
\[p_2^*\omega^{\kappa^{un}}/I_n\to p_1^!\omega^{\kappa^{un}}/I_n.\]
Recall that we have $X_0(\mathfrak p)^{\ord}_n=X_0(\mathfrak p)^{\ord,F}_n\coprod X_0(\mathfrak p)^{\ord,V}_n$. Moreover, the graph of $F$ with its two natural projections is $X_0(\mathfrak p)^{\ord,F}$ with $p_1$ and $p_2$. With this, we construct the correspondence in the following way, on $X_0(\mathfrak p)^{\ord,V}_n$ we define it as $0$, and on $X_0(\mathfrak p)^{\ord,F}_n$ as $F:F^*\omega^{\kappa^{un}}\to \omega^{\kappa^{un}}$, where we are using the graph identification just mentioned.

To continue, we express $H^1_c$ in terms of the pushing forward from $X_n^{\ord}$ to $X_n$, as in Section \ref{subsec:corrmodp}. We let $\mathcal I \subset \mathcal{O}_{X_n}$
be a locally principal sheaf of ideals so that the complement of the corresponding closed subscheme is $X^{\ord}_n$. We take a coherent sheaf $\mathcal F$ over $X_n$ extending $\omega^{\kappa^{un}}/I_n$. We may assume that $\mathcal{F}$ is $\mathcal{I}$-torsion free by replacing $\mathcal{F}$ with its quotient by the subsheaf of $\mathcal{I}$-power torsion. Then, the multiplication map $\mathcal{I}\otimes \mathcal{O}_{X_n}\mathcal{F} \to  \mathcal{I}^l\mathcal{F}$
is an isomorphism for all $l\geq 0$, and we use this to make sense of $\mathcal{I}^l\mathcal{F}$ for all $l \in \mathbb Z$. Then, if $j\colon X^{\ord}_n\to X$ denotes the inclusion, we have $j_*\omega^{\kappa^{un}}/I_n = {\rm colim}_l \mathcal I^l\mathcal F$. With this, we have the following identification $\operatorname{H}^1_c(X^{\ord}_{n},\omega^{\kappa^{un}}/I_n)=\lim_l \operatorname{H}^1(X_{n},\mathcal I^l\mathcal F)$.

Now, we continue with the construction of $F$: By pushing forward via $j$ we obtain
\[{\rm colim}_lp_2^*(\mathcal{I}^{-l}\mathcal{F})\to{\rm colim}_lp_1^!(\mathcal{I}^{-l}\mathcal{F}). \] 
It follows that there exists some $c$ for which there is a map $p_2^*\mathcal{F}\to p_1^!(\mathcal{I}^{-c}\mathcal{F})$. Furthermore, we note that $p^*_1\mathcal{I}$ and $p^*_2\mathcal{I}$ can be identified with $p^{-1}(\mathcal I)\mathcal O_{X_0(\mathfrak p)_n}$ and $p^{-1}_2(\mathcal I)\mathcal O_{X_0(\mathfrak p)_n}$, and so we view them as locally principal sheaves of ideals in $\mathcal O_{X_0(\mathfrak p)_n}$. They define the same closed subset of $X_0(\mathfrak p)_n$ (the complement of $X_0(\mathfrak p)^{\ord}_n)$, and so in particular we can pick some $m$ with $p_2^*\mathcal{I}^m \subset p_1^*\mathcal{I}$. Using this inclusion  we deduce a map
\[p_2^*(\mathcal{I}^{lm})\to p_1^!(\mathcal{I}^{-c+l}\mathcal{F}),\]
for all $l\geq0$. Taking cohomology we obtain a map $F:\operatorname{H}^1(X_n,\mathcal{I}^{lm}\mathcal{F})\to \operatorname{H}^1(X_n,\mathcal{I}^{-c+l}\mathcal{F})$. Passing to the limit, we get our endormorphism $F$ of $\operatorname{H}^1_c(X_n^{\ord},\omega^{\kappa^{un}}/I_n)$.

The next point is to check that $F$ is locally finite. For this, we do an induction argument. The case $n=1$ follows from the analogous properties form the Drinfeld modular curve $X$. For the induction step we can use the exact sequence
   
\[
 \resizebox{\hsize}{!}{$0\to \operatorname{H}_c^1(X_n^{\ord},\omega^{\kappa^{un}}\otimes (I_n/I_{n+1})) \to \operatorname{H}_c^1(X_{n+1}^{\ord},\omega^{\kappa^{un}}\otimes (\Lambda/I_{n+1})) \to \operatorname{H}_c^1(X_{n}^{\ord},\omega^{\kappa^{un}}\otimes (\Lambda/I_{n}))\to 0.$}
\]
 Now, using exactness properties of the notion of locally finite, one can deduce the properties for $n+1$ from $n$ and $\operatorname{H}_c^1(X_n^{\ord},\omega^{\kappa^{un}}\otimes (I_n/I_{n+1}))$.

 To prove the control theorem, note that by functoriality of $f_k$ we have a natural map 
\[\operatorname{H}^1_c(\mathfrak X^{\ord},\omega^{\kappa^{un}}) \otimes_{\Lambda,f_k}A_\fp \rightarrow \operatorname{H}^1_c(\mathfrak X^{\ord},\omega^k),
\]
both modules are $\varpi$-torsion free, so it is enough to prove the results modulo $\varpi$ and invoke Nakayama's lemma. But the result modulo $\varpi$ is exactly Corollary \ref{cor:descent}.
\end{proof}

\subsection{Duality for families} \label{sec:dualityfamilies}
We define the ``dualizing complex''  of $\omega^{\kappa^{un}}$ by 
 \[\Omega^1_{X/A_\fp}\otimes \operatorname{\underline{Hom}}(\omega^{\kappa^{un}},\Lambda\widehat\otimes \mathcal{O}_{\mathfrak X^{\ord}}).\]
We denote it by ${\bf D}(\omega^{\kappa^{un}})$. By Serre duality we already have a pairing 
\[\operatorname{H}^0(\mathfrak X^{\ord}, \omega^{\kappa^{un}}) \times \operatorname{H}^1_c(\mathfrak X^{\ord}, {\bf D}(\omega^{\kappa^{un}})) \to \Lambda.\]
Using the  Kodaira--Spencer isomorphism  and $\omega\cong\omega_D$ over $\mathfrak X^{\rm ord}$, we have that
\[{\bf D}(\omega^{\kappa^{un}})\cong \omega\otimes \omega_D(-2D)\otimes \operatorname{\underline{Hom}}(\omega^{\kappa^{un}},\Lambda\widehat\otimes \mathcal{O}_{\mathfrak X^{\ord}})\cong\omega^2(-2D)\otimes \operatorname{\underline{Hom}}(\omega^{\kappa^{un}},\Lambda\widehat\otimes \mathcal{O}_{\mathfrak X^{\ord}}).\]
 Moreover, the character $A_\mathfrak p^\times\to \Lambda^\times, t\mapsto t^2(\kappa^{un}(t))^{-1}$ induces an automorphism $d\colon \Lambda \to \Lambda$ and we have an isomorphism of $\mathcal{O}_{\mathfrak X^{\ord}}\widehat\otimes \Lambda$-modules ${\bf D}(\omega^{\kappa^{un}})\cong\omega^{\kappa^{un}}\otimes_{\Lambda,d}\Lambda$. 

\begin{defi}
We define the Atkin--Lehner involution $w$ of level $\mathfrak p$ to be the map on $X_0(\mathfrak p)$ sending the universal isogeny $\pi$ to $\pi^D$. 
\end{defi}

Note that $w$ corresponds to the adelic Hecke operator associated with $\left( 
 \begin{array}{cc}
    0  & 1 \\
     \varpi & 0
 \end{array}\right)$ and the diamond operator $\langle \varpi \rangle$ (defined in Section \ref{subsec:FU}) corresponds to $\left( 
 \begin{array}{cc}
    \varpi  & 0 \\
     0 & \varpi
 \end{array}\right)$ as defined in \cite[\S 6]{PinkSatake}. It easy to see that $w^2$ sends $(E,\pi)$ to $(\langle \varpi \rangle E, \pi)$.

\begin{thm}\label{thm:dualityfamilies}
    There is a locally finite operator $F$ of $\operatorname{H}^1_c(\mathfrak X^{\ord}, {\bf D}(\omega^{\kappa^{un}}))$ with  projector $e(F)$. The restriction of the pairing above to
    \[e(U_\mathfrak p)\operatorname{H}^0(\mathfrak X^{\ord}, \omega^{\kappa^{un}}) \times e(\langle \varpi^{-1} \rangle F)\operatorname{H}^1_c(\mathfrak X^{\ord}, {\bf D}(\omega^{\kappa^{un}}))\] 
    gives a perfect pairing.
\end{thm}
Note that the ordinary projector $e(\langle \varpi^{-1} \rangle F)$ and $e(F)$ are the same, as $\langle \varpi^{-1} \rangle$ acts via a finite order character.\footnote{Compare with the classical Atkin--Lehner operator of level $N$ for elliptic modular forms, that changes the Nebentypus to its complex conjugate, and the classical Petersson pairing which pairs forms of a given Nebentypus with the ones with complex conjugate Nebentypus.}
\begin{proof} Using ${\bf D}(\omega^{\kappa})
\cong \omega^{\kappa^{un}}\otimes_{\Lambda,d}\Lambda$, we define $F$ as a cohomological correspondence by defining it on the first term on the right hand side exactly as $F$ in Theorem \ref{thm:projandinter}. Furthermore, since $F$ is locally finite on $\omega^{\kappa^{un}}$, it is also locally finite on ${\bf D}(\omega^{\kappa})$, as it preserves the cuspidality condition, thus, we have the associated projector $e(F)$.

We now follow very closely the arguments in \cite[Section 4.3]{BoPi2022}. Applying the specialization map for all $k$ in Proposition \ref{prop:spk}, we have the following: 
\[
\begin{tikzcd}
    \operatorname{H}^0(\mathfrak X^{\ord}, \omega^{\kappa^{un}}) \times \operatorname{H}^1_c(\mathfrak X^{\ord}, {\bf D}(\omega^{\kappa^{un}})) \ar[r]\ar[d]& \Lambda\ar[d]\\
    \prod_{k\in \mathbb{Z}}\operatorname{H}^0(\mathfrak X^{\ord}, \omega^{k}) \times \operatorname{H}^1_c(\mathfrak X^{\ord},\omega^{1-k}\otimes\omega_D(-2D))  \ar[r]& A_\mathfrak p.
\end{tikzcd}
\]

We are left to relate ${\bf D}(F)$ and $U_\mathfrak p$.  We can use the density result from Proposition \ref{prop:densek} and reduce ourselves to proving that ${\bf D}(F)$ and $U_\mathfrak p$ almost coincide at classical weights.
We specialize Proposition \ref{prop:DualTp}, that says that ${\bf D}(T^{\text{naive}}_\mathfrak p) =\pi_{1-k}^{-1}\otimes {\rm tr}_{p_2}\otimes\pi^D_1 $, to the ordinary locus and we get
\[
{\bf D}(T^{\text{naive}}_\mathfrak p) = [\varpi]^{k-1} \pi^D_k\otimes {\rm tr}_{p_2} 
\]
where we used $\pi_{1-k}^{} \pi_{1-k}^{D} = [\varpi]^{1-k}$. 
Now, we need to compare $\pi^D_k\otimes {\rm tr}_{p_2}$ with $T^{\text{naive}}_{\mathfrak p}$: recall that the Atkin--Lehner map $w$ swaps $E$ to $E^D$ and the universal isogeny $\pi$ with $\pi^D$. From a direct linear algebra calculation on the torsion of the universal Drinfeld module, we get 
\[
p_1 w = p_2, \;  p_2w= \langle \varpi \rangle p_1.
\]
where we used that  $w^2$ amounts to the diamond operator $\langle \varpi \rangle$ on the prime-to-$\mathfrak{p}$ level structure and that $(E^D)^D \cong E$ \cite[Theorem 1.1(b)]{PaRa2003}. 
Hence, via the pullback by $w$, $\pi^D_k\otimes {\rm tr}_{p_2}$ acts on cohomology as \[
\langle \varpi \rangle^{-1} [\varpi]^{k-1} T^{\text{naive}}_\mathfrak p.
\] To compare with the normalized $T_{\mathfrak p}$ note that if $k<0$ we multiply by $[\varpi]^{-k}$, which simplifies $[\varpi]^{k-1}$ leaving $[\varpi]^{-1}$. 

Also, note that $w$ swaps $\fX^{\ord, F} $ with $\fX^{\ord, V}$, and we can conclude 
\[
{\bf D}(F)= \langle \varpi \rangle^{-1}   U_\mathfrak p.
\]

Finally, using the calculation above for ${\bf D}(F)$ and Corollary \ref{cor:classic}, we have that the following commutative diagram for any $k\geq 3$:
\[
\begin{tikzcd}
    e(U_\mathfrak p)\operatorname{H}^0(\mathfrak X^{\ord}, \omega^{k}) \times e(F)\operatorname{H}^1_c(\mathfrak X^{\ord},\omega^{1-k}\otimes\omega_D(-2D))  \ar[d]\ar[r]& A_\mathfrak p,\\
     e(T_\mathfrak p)\operatorname{H}^0(X, \omega^{k}) \times  e(T_\mathfrak p)\operatorname{H}^1(X,\omega^{1-k}\otimes\omega_D(-2D)) \ar[ru] & 
\end{tikzcd}
\]
where the bottom part of the diagram corresponds to classical Serre duality. 
\end{proof}


\end{document}